\documentclass{article}%
\usepackage{graphicx}
\usepackage{amsmath}
\usepackage{amsfonts}
\usepackage{amssymb}%
\providecommand{\U}[1]{\protect\rule{.1in}{.1in}}
\newtheorem{theorem}{Theorem}

\newtheorem{lemma}[theorem]{Lemma}

\newtheorem{proposition}[theorem]{Proposition}
\newtheorem{remark}[theorem]{Remark}

\newenvironment{proof}[1][Proof]{\textbf{#1.} }{\ \rule{0.5em}{0.5em}}
\begin{document}

\title{\textbf{Equivariant prequantization bundles on the space of connections and
characteristic classes} }
\author{\textsc{Roberto Ferreiro P\'{e}rez}\\Departamento de Econom\'{\i}a Financiera y Contabilidad I\\Facultad de Ciencias Econ\'omicas y Empresariales\\Universidad Complutense de Madrid\\Campus de Somosaguas, 28223-Pozuelo de Alarc\'on, Spain\\\emph{E-mail:} \texttt{roferreiro@ccee.ucm.es}}
\date{}
\maketitle

\begin{abstract}
We show how characteristic classes determine equivariant prequantization
bundles over the space of connections on a principal bundle. These bundles are
shown to generalize the Chern-Simons line bundles to arbitrary dimensions. Our
result applies to arbitrary bundles, and it is studied the action of both the
gauge group and the automorphisms group. The action of the elements in the
connected component of the identity of the group generalizes known results in
the literature. The action of the elements not connected with the identity is
shown to be determined by a characteristic class by using differential
characters and equivariant cohomology. We extend our results to the space of
Riemannian metrics and the actions of diffeomorphisms. In dimension 2, a
$\Gamma_{M}$-equivariant prequantization bundle of the Weil-Petersson
symplectic form on the Teichm\"{u}ller space is obtained, where $\Gamma_{M}$
is the mapping class group of the surface $M$.

\end{abstract}

\bigskip

\noindent\emph{Mathematics Subject Classification 2100:\/} Primary 53C05;
Secondary 53C08, 70S15, 58D27.

\medskip


\medskip

\noindent\emph{Acknowledgments:\/} Supported by \textquotedblleft Proyecto de
investigaci\'{o}n Santander-UCM PR26/16-20305\textquotedblright.

\section{Introduction}

In this paper we study the relationship between characteristic classes and
equivariant prequantization line bundles over the space of connections. We
recall two classical examples of this relation (see Section \ref{notation}%
\ for the notation).

In the first example, let $\Sigma$ be a closed (i.e. compact and without
boundary) oriented surface, $P=\Sigma\times SU(2)$ the trivial principal
$SU(2)$-bundle, and $p\in I_{\mathbb{Z}}^{2}(SU(2))$ the polynomial associated
to the second Chern class. We denote by $\mathcal{A}$ and $\widetilde
{\mathcal{A}}$ the spaces of connections and irreducible connections on $P$.
In \cite{AB1} Atiyah and Bott show that this polynomial determines a
symplectic structure $\sigma$ on the space of connections $\mathcal{A}$ which
is invariant under the action of the group $\mathcal{G}$ of gauge
transformations. Moreover the curvature map determines a moment map $\mu$ for
$\sigma$. By symplectic reduction a symplectic structure $\underline{\sigma}%
$\ on the moduli space of irreducible flat connections $\widetilde
{\mathcal{F}}/\mathcal{G}$ is obtained. Furthermore, in \cite{RSW} it is shown
that the action of $\mathcal{G}$ admits a lift to $\widetilde{\mathcal{A}%
}\times U(1)$ by $U(1)$-bundle automorphisms hence defining a $\mathcal{G}%
$-equivariant $U(1)$-bundle over $\widetilde{\mathcal{A}}$ (or what is
equivalent, a $\mathcal{G}$-equivariant Hermitian line bundle). By taking the
quotient, they obtain an Hermitian line bundle $\mathcal{L}\rightarrow
\widetilde{\mathcal{F}}/\mathcal{G}$ (which is proved to be isomorphic to the
Quillen determinant line bundle) and a natural connection on $\mathcal{L}$
whose curvature is $\underline{\sigma}$. We recall that all this constructions
can be done based only on the polynomial $p$.

The second example is the classical $3$-dimensional Chern-Simons theory. Let
$M$ be a compact $3$-dimensional manifold and $P=M\times SU(2)$ the trivial
principal $SU(2)$-bundle. For simplicity we assume that $\mathcal{G}$\ is a
group that acts freely by gauge transformations on $\mathcal{A}$ and that
$\mathcal{A}\rightarrow\mathcal{A}/\mathcal{G}$ is a principal $\mathcal{G}%
$-bundle. If $M$ is closed then the Chern-Simons action associated to a
polynomial $p\in I_{\mathbb{Z}}^{2}(SU(2))$ determines a $\mathcal{G}%
$-invariant function $\mathcal{A}\rightarrow\mathbb{R}/\mathbb{Z}$ and hence a
function on the quotient $\mathcal{A}/\mathcal{G}\rightarrow\mathbb{R}%
/\mathbb{Z}$. However, when $M$ is a manifold with boundary $\partial M$ the
Chern-Simons action is not a function on $\mathcal{A}/\mathcal{G}$, but it
determines a section of a line bundle $\mathcal{L}_{\partial M}\rightarrow
\mathcal{A}/\mathcal{G}$ called the Chern-Simons line (see e.g. \cite{Freed1}%
). Again all the constructions are based on a polynomial $p$. However, as
pointed out in \cite{DW}, to determine the Chern-Simons action for nontrivial
bundles it is also necessary to choose a universal characteristic class
$\Upsilon\in H^{4}(\mathbf{B}G)$.

We generalize these two examples to arbitrary bundles, groups and dimensions
in the following way. We recall (see \cite{equiconn}) that if $P\rightarrow M$
is a principal $G$-bundle and $\mathcal{A}$ the space of connections on $P$,
the principal $G$-bundle $\mathbb{P}=P\times\mathcal{A}\rightarrow
M\times\mathcal{A}$ admits a canonical (or tautological) connection
$\mathbb{A}$ which is invariant under the action of the group $\mathrm{Aut}P$
of automorphisms of $P$. If a group $\mathcal{G}$ acts on $P\rightarrow M$ by
gauge transformations, then for any invariant polynomial $p\in I_{\mathbb{Z}%
}^{r}(G)$ we can consider the $\mathcal{G}$-equivariant characteristic forms
$p_{\mathcal{G}}^{\mathbb{A}}\in\Omega_{\mathcal{G}}^{2r}(M\times\mathcal{A}%
)$\ of $\mathbb{A}$. If $c$ is a closed oriented $d$-dimensional submanifold
of $M$, by integrating $p_{\mathcal{G}}^{\mathbb{A}}$ over $c$, we obtain
$\int_{c}p_{\mathcal{G}}^{\mathbb{A}}\in\Omega_{\mathcal{G}}^{2r-d}%
(\mathcal{A})$ which is closed for the Cartan differential $D$. When $d=2r-2$,
$\varpi_{c}=\int_{c}p_{\mathcal{G}}^{\mathbb{A}}\in\Omega_{\mathcal{G}}%
^{2}(\mathcal{A})$ is a closed equivariant $2$-form, i.e., $\varpi_{c}%
=\sigma_{c}+\mu_{c}$ where $\sigma_{c}$ is a closed $\mathcal{G}$-invariant
$2$-form and $\mu_{c}$ a co-moment map for $\sigma_{c}$. Our main result is
the following

\begin{theorem}
\label{TheoremAIntro}Let $c$ be a closed submanifold of dimension $2r-2$ of
$M$, $p\in I_{\mathbb{Z}}^{r}(G)$, $\Upsilon\in H^{2r}(\mathbf{B}%
G,\mathbb{Z})$ a characteristic class compatible with $p$ (i.e., they
determine the same real characteristic class) and $A_{0}$ a background
connection\ on $P$. These data determine a lift of the action of $\mathcal{G}$
on $\mathcal{A}$ to an action on $\mathcal{U}_{c}=\mathcal{A}\times
U(1)\rightarrow\mathcal{A}$ by $U(1)$-bundle automorphisms, and a
$\mathcal{G}$-invariant connection form $\Xi_{c}$ such that the $\mathcal{G}%
$-equivariant curvature of $\Xi_{c}$ is\textrm{ }$\varpi_{c}$.
\end{theorem}

Due to the equivalence between principal $U(1)$-bundles and Hermitian line
bundles, we also obtain a $\mathcal{G}$-equivariant Hermitian line bundle
$\mathcal{L}_{c}\rightarrow\mathcal{A}$ with connection $\nabla^{\Xi_{c}}$.
Our result also generalizes the Chern-Simons line as we prove the following result.

\begin{proposition}
\label{deltauintro}If $c=\partial u$ for some $u\subset M$, then
$S_{u}(A)\!=\!\exp(-2\mathrm{\pi}i\cdot\!%
{\textstyle\int\nolimits_{u}}
Tp(A,A_{0}))$ determines a$\ \mathcal{G}$-invariant section of $\mathcal{U}%
_{c}\rightarrow\mathcal{A}$, or what it is equivalent, a $\mathcal{G}%
$-invariant\ section of unitary norm of $\mathcal{L}_{c}\rightarrow
\mathcal{A}$.
\end{proposition}

Thus $p$, $\Upsilon,$ $c$ and $A_{0}$ determine a $\mathcal{G}$%
-equivariant\ prequantization bundle for $(\mathcal{A},\varpi_{c})$. If we
change the background connection $A_{0}$ we obtain a different connection, and
a different action, but we prove that there exists a canonical $\mathcal{G}%
$-equivariant isomorphisms between them. Therefore we can consider that
different background connections $A_{0}$ determine different global
trivializations of the same prequantization bundle, and hence that it only
depends on $p,\Upsilon$ and $c$.

We prove that for any $X\in\mathrm{Lie}\mathcal{G}$, its lift $X_{\mathcal{U}%
_{c}}\in\mathfrak{X}(\mathcal{U}_{c})$ does not depend on $\Upsilon$. Hence
neither does the action of the connected component of the identity
$\mathcal{G}_{0}$. Nevertheless the action of the elements of $\mathcal{G}$
which are not connected with the identity depends on $\Upsilon$. For certain
groups there is a bijection $I_{\mathbb{Z}}^{r}(G)\simeq H^{2r}(\mathbf{B}%
G,\mathbb{Z})$ and in these cases the action is determined only by $c$, $p$
and $A_{0}$. This happens for example in the case $G=U(n)$ as $H^{\bullet
}(\mathbf{B}U(n),\mathbb{Z})\simeq\mathbb{Z}[c_{1},\ldots,c_{n}]$ where
$c_{1},\ldots,c_{n}$ are the Chern classes (e.g. see \cite[Chapter 23.7]%
{May}). But for a general group $G$ the cohomology $H^{2r}(\mathbf{B}%
G,\mathbb{Z})$ may contain torsion elements, and $\Upsilon$ is not determined
by $p$. In that cases non-equivalent actions can exist with the same $c$, $p$
and $A_{0}$ if $\mathcal{G}$ is not connected.

We study the dependence on $c$. It can be better understood in terms on the
Hermitian line bundle $\mathcal{L}_{c}\rightarrow\mathcal{A}$. If $-c$ denotes
the submanifold $c$ with the opposed orientation, then we have $\mathcal{L}%
_{-c}=\mathcal{L}_{c}^{\ast}$, and if$\ c^{\prime}$ is another closed oriented
submanifold then $\mathcal{L}_{c+c^{\prime}}\simeq\mathcal{L}_{c}%
\otimes\mathcal{L}_{c^{\prime}}$. In particular, if $\partial u=c-c^{\prime}$
by Proposition \ref{deltauintro} $S_{u}$ determines a section of unitary norm
on $\mathcal{L}_{c-c^{\prime}}=\mathcal{L}_{c}\otimes\mathcal{L}_{c^{\prime}%
}^{\ast}\simeq\mathrm{Hom}(\mathcal{L}_{c^{\prime}},\mathcal{L}_{c})$ which is
an isomorphism.

In Section \ref{Irreducible}\ the restriction of the prequantization bundle
$\mathcal{U}_{c}$ to the space of irreducible connections $\widetilde
{\mathcal{U}}_{c}=\widetilde{\mathcal{A}}\times U(1)$ is studied. We have a
well defined quotient manifold $\widetilde{\mathcal{A}}/\mathcal{G}$, and for
the trivial $SU(2)$-bundle over a Riemann surface we have also a well defined
quotient $U(1)$-bundle $\widetilde{\mathcal{U}}_{c}/\mathcal{G}\rightarrow
\widetilde{\mathcal{A}}/\mathcal{G}$. If $p$ is the second Chern polynomial
and $c=M,$ then $\sigma_{c}$ and $\mu_{c}$ coincide with the Atiyah-Bott
symplectic structure and moment map (see \cite{equiconn}). The connection
$\Xi_{c}$ does not project onto a connection on $\widetilde{\mathcal{U}}%
_{c}/\mathcal{G}$ as $\iota_{X_{\mathcal{U}_{c}}}\Xi_{c}=-\mu_{c}(X)$ for
$X\in\mathrm{Lie}\mathcal{G}$. However, if $\widetilde{\mathcal{F}}$ is the
space of irreduclible flat conections we have $\widetilde{\mathcal{F}}%
\subset\mu_{c}^{-1}(0)$, and the restriction of $\Xi_{c}$ to $\widetilde
{\mathcal{F}}\times U(1)$ is $\mathcal{G}$-basic and projects onto a
connection $\underline{\Xi}_{c}$ on $(\widetilde{\mathcal{F}}\times
U(1))/\mathcal{G}\rightarrow\widetilde{\mathcal{F}}/\mathcal{G}$. Furthermore
the curvature of $\underline{\Xi}_{c}$ is the form $\underline{\sigma}_{c}$
obtained by symplectic reduction of $(\mathcal{A},\sigma_{c},\mu_{c})$. Hence
our result generalizes that of \cite{RSW}. Furthermore, we also show that for
other groups and bundles, the prequantization bundle of $\widetilde
{\mathcal{A}}/\mathcal{G}$ is not determined by the characteristic classes of
$G$, but by those of the group $\widetilde{G}=G/Z(G)$, where $Z(G)$ is the
center of $G$.

The symmetry group usually\ considered in physical theories is the group of
gauge transformations. However sometimes it is necessary to consider the lift
of the action of the automorphism group $\mathrm{Aut}P$ to $\mathcal{U}_{c}$
(see for example \cite{Andersen1,Andersen} and references therein). We show
that Theorem \ref{TheoremAIntro} is also valid when $\mathcal{G}$ acts on $P$
by automorphisms preserving the orientation of $M$\ in the following cases:

- $M$ is a closed oriented manifold of dimension $d=2r-2$ and $c=M$.

- $M$ is a compact oriented manifold of dimension $d=2r-1$ with boundary
$\partial M$ and $c=\partial M$. In this case Proposition \ref{deltauintro} is
also valid.

Finally we apply our results to the space $\mathfrak{Met}M$ of Riemannian
metrics and the action of the orientation preserving diffeomorphisms
$\mathrm{Diff}^{+}M$. For closed manifolds of dimension $4r-2$ the integer
combinations of Pontryagin classes of degree $r$ determine $\mathrm{Diff}%
^{+}M$-equivariant prequantization bundles of the presymplectic structures
defined in \cite{WP}. In particular, for a surface,\ the first Pontryagin
class is shown to determine a canonical holomorphic prequantization bundle for
the Teichm\"{u}ller space endowed with the Weil-Petersson symplectic form.
This bundle is shown to be equivariant with respect of the action of the
mapping class group of the surface. Furthermore, for compact manifolds of
dimension $4r-1$ with boundary, we obtain Chern-Simons line bundles for
Riemannian metrics.

Let us explain the way in which Theorem \ref{TheoremAIntro} is obtained. For
simplicity we assume that $\mathcal{G}$ is the group of gauge transformations
that fixes a point $p_{0}\in P$. As it is well known $\mathcal{G}$ acts freely
on $\mathcal{A}$ and we have well defined quotient manifolds. In \cite{AS}
Chern-Weil theory is applied to the principal $G$-bundle $(P\times
\mathcal{A})/\mathcal{G}\rightarrow M\times\mathcal{A}/\mathcal{G}$. Moreover,
if $p\in I_{\mathbb{Z}}^{r}(G)$, the Chern-Simons construction can also be
applied to this bundle. This is done in \cite{flat} by using the
Cheeger-Simons approach of \cite{Cheeger} based on differential characters.
The space of differential characters of order $k$ on $M$ is denoted by
$\hat{H}^{k}(M)$. If $\mathfrak{A}$ is a connection on the principal
$\mathcal{G}$-bundle $\mathcal{A}\rightarrow\mathcal{A}/\mathcal{G}$, it
determines a connection $\underline{\mathfrak{A}}$ on $(P\times\mathcal{A}%
)/\mathcal{G}\rightarrow M\times\mathcal{A}/\mathcal{G}$ (see \cite{flat} for
details). Hence, if $\Upsilon\in H^{2r}(\mathbf{B}G)$ is a universal
characteristic class compatible with $p$, there exists a differential
character (the Chern-Simons differential character) $\chi_{\underline
{\mathfrak{A}}}\in\hat{H}^{2r}(M\times\mathcal{A}/\mathcal{G})$ whose
curvature is $p(F_{\underline{\mathfrak{A}}})$. By integration over a
submandifold $c$ a differential character $\int_{c}\chi_{\underline
{\mathfrak{A}}}\in\hat{H}^{2r-d}(\mathcal{A}/\mathcal{G})$ is obtained, where
$d=\dim C$. In \cite{flat}, by appliying our results of Section \ref{2orden},
the characters of order $2$ are interpreted geometrically as the holonomy of a
connection on a $U(1)$-bundle $\underline{\mathcal{U}}_{c}\rightarrow
\mathcal{A}/\mathcal{G}$. We generalize this construction to non-free actions,
to the action of automorphisms and also to the space of Riemannian\ metrics.

It is possible to extend the construction of \cite{flat} to non-free actions
by using equivariant cohomology. We do not follow this approach because it
requires the use of connections and quotients on principal $\mathcal{G}%
$-bundles for infinite dimensional groups. This can be technically dificult,
especially if we want to apply it to groups of automorphisms and
diffeomorphisms that should be considered as Fr\'{e}chet Lie groups. To avoid
this problem, we give\ a direct definition of the lift of the action of each
element $\phi\in\mathcal{G}$. We show that it can be done using only the
action of discrete groups and it does not require the use of quotients and
auxiliary connections for infinite dimensional groups. We also prove that the
bundle constructed in \cite{flat} coincides with our bundle.

\section{Notations and conventions\label{notation}}

In this paper we consider two Lie groups. The group $G$ is the structure group
of a principal bundle $P\rightarrow M\,$and it is supposed to be finite
dimensional and with a finite number of connected components (in order to
apply Chern-Simons construction). The second group $\mathcal{G}$ is a symmetry
group (usually infinite dimensional), and $\mathcal{G}_{0}$ denotes the
connected component of the identity on $\mathcal{G}$.

We denote by $I_{\mathbb{Z}}^{r}(G)$ the set of $G$-invariant polynomials on
its Lie algebra $\mathfrak{g}$ whose characteristic classes have integral
periods. We denote by $\mathbf{E}G\rightarrow\mathbf{B}G$ a universal
principal $G$-bundle. A polynomial $p\in I_{\mathbb{Z}}^{r}(G)$ and a
characteristic class $\Upsilon\in H^{2r}(\mathbf{B}G,\mathbb{Z})$ are said to
be compatible if they determine the same real characteristic class. We denote
by $\mathcal{I}_{\mathbb{Z}}^{r}=\{(p,\Upsilon)\in I_{\mathbb{Z}}^{r}(G)\times
H^{2r}(\mathbf{B}G,\mathbb{Z})\colon$ $p,\Upsilon$ are compatible$\}$, and by
$\Upsilon_{P}$ the characteristic class of $P\rightarrow M$ associated to
$\Upsilon$.

The Maurer Cartan form of $U(1)$ is denoted by $\theta=u^{-1}du$, and
$\partial_{\theta}\in\mathfrak{X}(U(1))$ is the vector field such that
$\theta(\partial_{\theta})=1$. If $\pi\colon\mathcal{U}\rightarrow N$ is a
principal $U(1)$ bundle and $\Xi\in\Omega^{1}(\mathcal{U},i\mathbb{R)}$ is a
connection then the curvature form$\ \mathrm{curv}(\Xi)\in\Omega^{2}(N)$ is
defined by the property $\pi^{\ast}(\mathrm{curv}(\Xi))=\frac{i}{2\pi}d\Xi$.
The log-holonomy $\log\mathrm{hol}_{\Xi}(\gamma)\in\mathbb{R}/\mathbb{Z}$ of
$\Xi$ on a closed curve $\gamma\colon I\rightarrow N$ with $\gamma
(0)=\gamma(1)$ is determined by the relation $\bar{\gamma}(1)=\bar{\gamma
}(0)\cdot\exp(2\pi i\log\mathrm{hol}_{\Xi}(\gamma))$, where $\bar{\gamma
}\colon I\rightarrow\mathcal{U}$ is the $\Xi$-horizontal lift of $\gamma$. The
(real) first Chern class of $\mathcal{U}$ is the cohomology class of
$\mathrm{curv}(\Xi)$. We denote by $I$ the interval $[0,1]$.

\section{Cheeger-Simons differential characters}

We recall the definition of differential characters (see \cite{Cheeger} and
\cite{BB}\ for details). We denote by $C_{k}(N)$ and $Z_{k}(N)$ the smooth
chains and cycles on $N$. A Cheeger-Simons differential character of order $k$
is a homomorphism $\chi\colon Z_{k-1}(N)\rightarrow\mathbb{R}/\mathbb{Z}$ such
that there exist $\alpha\in\Omega^{k}(N)$ with satisfies $\chi(\partial u)=%
{\textstyle\int\nolimits_{u}}
\alpha$ for every $u\in C_{k}(N)$. We say that $\chi$ is a differential
character with curvature $\mathrm{curv}(\chi)=\alpha$ and it can be proved
that $d\mathrm{curv}(\chi)=0$. We recall (e.g. see \cite{BB}) that $\chi
(u)\ $is invariant under reparametrizations, i.e. if $u\colon U\rightarrow N$,
and $\varphi$ is an orientation preserving diffeomorphism of $U$, then
$\chi(u\circ\varphi)=\chi(u)$. We denote the space of differential characters
of order $k$ on $N$ by $\hat{H}^{k}(N)$. We have a map $\mathrm{char}%
\colon\hat{H}^{k}(N)\rightarrow H^{k}(N,\mathbb{Z})$, and the class
$\mathrm{char}(\chi)$ is called the characteristic class of $\chi$. The maps
$\mathrm{char}$ and $\mathrm{curv}$ are compatible in the sense that we have
$r(\mathrm{char}(\chi))=[\mathrm{curv}(\chi)]\in H^{k}(N,\mathbb{R})$. If
$f\colon N^{\prime}\rightarrow N$ is a smooth map, it induces a map $f^{\ast
}\colon\hat{H}^{k}(N)\rightarrow\hat{H}^{k}(N^{\prime})$ defined by $f^{\ast
}\chi(u)=\chi(f\circ u)$. Given $\beta\in\Omega^{k-1}(N)$ we define a
differential character $\varsigma(\beta)\in\hat{H}^{k}(N)$ by setting
$\varsigma(\beta)(s)=\int_{s}\beta$ for $s\in Z_{k-1}(M)$. We have
$\mathrm{curv}(\varsigma(\beta))=d\beta$, and $\mathrm{char}(\varsigma
(\beta))=0$. Note that $\varsigma(d\alpha)(s)=\int_{s}d\alpha=\int_{\partial
s}\alpha=0$ as $\partial s=0$.

\subsection{Differential characters of order 2\label{2orden}}

First we recall that if $\mathcal{U}\rightarrow M$ a principal $U(1)$-bundle
with connection $\Theta$, and curvature $\omega\in\Omega^{2}(M)$, the
log-holonomy of $\Theta$ $\log\mathrm{hol}_{\Theta}\colon Z_{1}(M)\rightarrow
\mathbb{R}/\mathbb{Z}$ is a differential character with curvature
$\mathrm{curv}(\log\mathrm{hol}_{\Theta})=\omega$ and $\mathrm{char}%
(\log\mathrm{hol}_{\Theta})=c_{1}(\mathcal{U})$. Conversely, by a classical
result in differential cohomology,\ every second order\ differential character
can be represented\ as the holonomy of a connection $\Theta$ on a principal
$U(1)$ bundle $\mathcal{U}\rightarrow M$. The bundle $\mathcal{U}$ and the
connection $\Theta$ are determined by $\chi$ only modulo isomorphisms. In the
next Proposition we show that in the following restrictive equivariant case it
is possible to give a concrete bundle and connection

\begin{theorem}
\label{Prop}Let $N$ be a connected manifold with $H_{1}(N,\mathbb{Z})=0$ in
which $\mathcal{G}$ acts in such a way that $\pi\colon N\rightarrow
N/\mathcal{G}$ is a principal $\mathcal{G}$-bundle. Let $\chi\in\hat{H}%
^{2}(N/\mathcal{G})$ be a second order differential character on
$N/\mathcal{G}$ with curvature $\omega$, and assume that there
exists\ $\lambda\in\Omega^{1}(N)$ such that $\pi^{\ast}\omega=d\lambda$. Then
there exists a unique lift of the action of $\mathcal{G}$ to $N\times U(1)$ by
$U(1)$-bundle automorphism such that $\Theta=\theta-2\pi i\lambda\in\Omega
^{1}(N\times U(1),i\mathbb{R})$ is projectable onto a connection
$\underline{\Theta}$ on $\mathcal{U}=(N\times U(1))/\mathcal{G}\rightarrow
N/\mathcal{G}$ and $\chi=\log\mathrm{hol}_{\underline{\Theta}}$. The action of
$\phi\in\mathcal{G}$ on $N\times U(1)\mathbb{\ }$is given by $\Phi_{\phi
}\colon N\times U(1)\rightarrow N\times U(1),$ $\Phi_{\phi}(x,u)=(\phi
x,\exp(2\mathrm{\pi}i\alpha_{\phi}(x))\cdot u)$, where $\alpha_{\phi}\colon
N\rightarrow\mathbb{R}/\mathbb{Z}$ is defined by $\alpha_{\phi}(x)=\int
\nolimits_{\gamma}\lambda-\chi(\pi\circ\gamma)$, and $\gamma$ is any curve on
$N$ joining $x$ and $\phi x$.
\end{theorem}

The proof of Theorem \ref{Prop}\ is a consequence of the following lemmas. If
$\gamma$ is a curve on $N$, we denote by $\underline{\gamma}=\pi\circ\gamma$
the projected curve on $N/\mathcal{G}$.

\begin{lemma}
Let $\gamma$ and $\gamma^{\prime}$ be two curves on $N$ joining $x$ and $\phi
x$. Then $%
{\textstyle\int\nolimits_{\gamma}}
\lambda-\chi(\underline{\gamma})=%
{\textstyle\int\nolimits_{\gamma^{\prime}}}
\lambda-\chi(\underline{\gamma}^{\prime})$.
\end{lemma}

\begin{proof}
As $H_{1}(N,\mathbb{Z})=0$ we have $\gamma-\gamma^{\prime}=\partial D$ on $N$.
If $\underline{D}=\pi\circ D$, then $\underline{\gamma}-\underline{\gamma
}^{\prime}=\partial\underline{D}$, and hence $\chi(\underline{\gamma}%
)-\chi(\underline{\gamma}^{\prime})=\chi(\underline{\gamma}-\underline{\gamma
}^{\prime})=\int\nolimits_{\underline{D}}\omega=\int\nolimits_{D}d\lambda
=\int\nolimits_{\gamma}\lambda-\int\nolimits_{\gamma^{\prime}}\lambda$.
\end{proof}

For every $\phi\in\mathcal{G}$ we define $\alpha_{\phi}\colon N\rightarrow
\mathbb{R}/\mathbb{Z}$ by $\alpha_{\phi}(x)=\int\nolimits_{\gamma}\lambda
-\chi(\underline{\gamma})$, where $\gamma$ is a curve on $N$ joining $x$ and
$\phi x$ (it is well defined by the preceding Lemma).

\begin{lemma}
\label{cambiopunto}a) We have $\alpha_{\phi}(x^{\prime})=\alpha_{\phi}%
(x)+\int\nolimits_{\gamma_{xx^{\prime}}}(\phi^{\ast}\lambda-\lambda)$ for any
curve $\gamma_{xx^{\prime}}$ joining $x$ and $x^{\prime}$.

b) As a consequence of a) we have $d\alpha_{\phi}=\phi^{\ast}\lambda-\lambda$.

c) If $\phi_{t}$ is a $1$-parameter group on $\mathcal{G}$ with $\dot{\phi}%
=X$, then $\left.  \tfrac{d\alpha_{\phi_{t}}(x)}{dt}\right\vert _{t=0}%
=\lambda(X_{N})(x)$.
\end{lemma}

\begin{proof}
a) If $\gamma$ is a curve joining $x$ and $\phi x$, $\gamma_{x^{\prime}x}$ a
curve joining $x^{\prime}$ and $x$, and $\gamma_{xx^{\prime}}$ the inverse
curve. Then $\gamma^{\prime}=\gamma_{x^{\prime}x}\ast\gamma\ast\phi
\gamma_{xx^{\prime}}$ is a curve joining $x^{\prime}$ and $\phi x^{\prime}$.
Clearly $\underline{\gamma}^{\prime}=\underline{\gamma}$ on $Z_{1}%
(N/\mathcal{G})$, and hence $\chi(\underline{\gamma})=\chi(\underline{\gamma
}^{\prime})$. We have
\begin{align*}
\alpha_{\phi}(x^{\prime})  &  =%
{\textstyle\int\nolimits_{\gamma^{\prime}}}
\lambda-\chi(\underline{\gamma}^{\prime})=%
{\textstyle\int\nolimits_{\gamma_{x^{\prime}x}}}
\lambda+%
{\textstyle\int\nolimits_{\gamma}}
\lambda+%
{\textstyle\int\nolimits_{\phi\gamma_{xx^{\prime}}}}
\lambda-\chi(\underline{\gamma})\\
&  =\alpha_{\phi}(x)+%
{\textstyle\int\nolimits_{\gamma_{xx^{\prime}}}}
\phi^{\ast}\lambda-%
{\textstyle\int\nolimits_{\gamma_{xx^{\prime}}}}
\lambda=\alpha_{\phi}(x)+%
{\textstyle\int\nolimits_{\gamma_{xx^{\prime}}}}
(\phi^{\ast}\lambda-\lambda)
\end{align*}

b) If $\gamma$ is a curve with $\gamma(0)=x$ and $\gamma^{\prime}(0)=X\in
T_{x}N$, we have $\alpha_{\phi}(\gamma(s))=\alpha_{\phi}(x)+%
{\textstyle\int\nolimits_{0}^{s}}
(\phi^{\ast}\lambda-\lambda)_{\gamma(s)}(\gamma^{\prime}(s))ds$, and hence
$d\alpha_{\phi}(x)(X)=\left.  \frac{d\alpha_{\phi}(x_{s})}{ds}\right\vert
_{s=0}=(\phi^{\ast}\lambda-\lambda)_{x}(X)$.

c) Let $X\in\mathrm{Lie}\mathcal{G}$ and $\phi_{t}$ a $1$-parameter group on
$\mathcal{G}$ with $\phi_{0}=\mathrm{1}_{\mathcal{G}}$ and $\dot{\phi}_{0}=X$,
and we define $\gamma(t)=\phi_{t}x.$ We have $\underline{\gamma}=0$ on
$Z_{1}(N/\mathcal{G})$ and $\dot{\gamma}(0)=X_{N}(x)$. Then $\alpha_{\phi_{t}%
}(x)=%
{\textstyle\int\nolimits_{\gamma}}
\lambda-\chi(\underline{\gamma})=%
{\textstyle\int\nolimits_{\gamma}}
\lambda=%
{\textstyle\int\nolimits_{0}^{t}}
\lambda_{\phi_{s}x}(\dot{\gamma}(s))ds$ and $\left.  \tfrac{d\alpha_{\phi_{t}%
}(x)}{dt}\right\vert _{t=0}=\lambda(X_{N}(x)).$
\end{proof}

The action $\alpha$ satisfies the following cocycle condition

\begin{lemma}
\label{cociclo}We have $\alpha_{\phi_{2}\phi_{1}}(x)=\alpha_{\phi_{1}%
}(x)+\alpha_{\phi_{2}}(\phi_{1}x)$ for any $x\in N$, $\phi_{1},\phi_{2}%
\in\mathcal{G}$.
\end{lemma}

\begin{proof}
Let $\gamma_{1}$ be a curve joining $x$ and $\phi_{1}x$ and $\gamma_{2}$ be a
curve joining $\phi_{1}x$ and $\phi_{2}\phi_{1}x$. Then $\gamma^{\prime
}=\gamma_{2}\ast\gamma_{1}$ is a curve joining $x$ and $\phi_{2}\phi_{1}x$. We
have $\underline{\gamma}^{\prime}=\underline{\gamma}_{2}+\underline{\gamma
}_{1}$ on $Z_{1}(N/\mathcal{G})$. Hence%
\[
\alpha_{\phi_{2}\phi_{1}}(x)=%
{\textstyle\int\nolimits_{\gamma^{\prime}}}
\lambda-\chi(\underline{\gamma}^{\prime})=%
{\textstyle\int\nolimits_{\gamma_{2}}}
\lambda+%
{\textstyle\int\nolimits_{\gamma_{1}}}
\lambda-\chi(\underline{\gamma}_{1})-\chi(\underline{\gamma}_{2})=\alpha
_{\phi_{1}}(x)+\alpha_{\phi_{2}}(\phi_{1}x).
\]

\end{proof}

We define the action of $\mathcal{G}$ on $N\times U(1)$ by $\Phi_{\phi
}(x,u)=(\phi x,\exp(2\pi i\alpha_{\phi}(x))\cdot u)$. It defines a group
action as we have%
\begin{align*}
\Phi_{\phi_{2}}(\Phi_{\phi_{1}}(x,u))  &  =\Phi_{\phi_{2}}((\phi_{1}%
x,\exp(2\pi i\alpha_{\phi_{1}}(x))\cdot u))\\
&  =(\phi_{2}\phi_{1}x,\exp(2\pi i(\alpha_{\phi_{1}}(x)+\alpha_{\phi_{2}}%
(\phi_{1}x))\cdot u)\\
&  =(\phi_{2}\phi_{1}x,\exp(2\pi i\alpha_{\phi_{2}\phi_{1}}(x))\cdot
u)=(\Phi_{\phi_{2}\phi_{1}})(x,u))
\end{align*}

We also define $\Theta=\theta-2\pi i\lambda\in\Omega^{1}(N\times
U(1),i\mathbb{R})$, $\mathcal{U}=(N\times U(1))/\mathcal{G}$ and we denote by
$\overline{\pi}\colon N\times U(1)\rightarrow\mathcal{U}$ the projection. For
every $\phi\in\mathcal{G}$ we have
\[
\Phi^{\ast}\Theta=\Phi^{\ast}\theta-2\pi i\phi^{\ast}\lambda=(\theta+2\pi
id\alpha_{\phi})-2\pi i\phi^{\ast}\lambda=(\theta+2\pi i(\phi^{\ast}%
\lambda-\lambda))-2\pi i\phi^{\ast}\lambda=\Theta.
\]
Moreover, for every $X\in\mathrm{Lie}\mathcal{G}$, if $\phi_{t}$ is a curve on
$\mathcal{G}$ with $X=\dot{\phi}$ then we have $X_{N\times U(1)}=X_{N}+2\pi
i\left.  \tfrac{d\alpha_{\phi_{t}}(x)}{dt}\right\vert _{t=0}\partial_{\theta
}=X_{N}+2\pi i\lambda(X_{N})\partial_{\theta}$, and hence $\Theta(X_{N\times
U(1)})=0$. We conclude that $\Theta$ is a $\mathcal{G}$-basic form, i.e. there
exists $\underline{\Theta}\in\Omega^{1}(N/\mathcal{G},i\mathbb{R})$ such that
$\overline{\pi}^{\ast}\underline{\Theta}=\Theta$. Clearly $\Theta$ is a
connection form.

\begin{lemma}
We have $\log\mathrm{hol}_{\underline{\Theta}}=\chi$.
\end{lemma}

\begin{proof}
Given a loop $\underline{\gamma}$ on $N/\mathcal{G}$ with $\underline{\gamma
}(0)=\underline{\gamma}(1)=[x]\in N/\mathcal{G}$, we can find a curve $\gamma$
on $N$ with $\gamma(0)=x$ such that $\pi\circ\gamma=\underline{\gamma}$. We
have $\gamma(1)=\phi x$ for some $\phi\in\mathcal{G}$. The $\Theta$-horizontal
lift of $\gamma$ to $N\times U(1)$ starting at the point $(x,0)$ is given by
$\overline{\gamma}(s)=(\gamma(s),\exp(2\pi i%
{\textstyle\int\nolimits_{0}^{s}}
\lambda_{\gamma(t)}(\dot{\gamma}(t))dt))$. The curve $\overline{\pi}%
\circ\overline{\gamma}$ is a $\underline{\Theta}$-horizontal lift to
$\mathcal{U}$ of the loop $\underline{\gamma}=\pi\circ\gamma$. In particular
we have $\overline{\pi}\circ\overline{\gamma}(1)=(\phi x,\exp(2\pi i%
{\textstyle\int\nolimits_{\gamma}}
\lambda))\sim_{\mathcal{G}}(x,\exp(2\pi i(%
{\textstyle\int\nolimits_{\gamma}}
\lambda-\alpha_{\phi}(x)))$. Hence $\log\mathrm{hol}_{\underline{\Theta}%
}(\underline{\gamma})=%
{\textstyle\int\nolimits_{\gamma}}
\lambda-\alpha_{\phi}(x)=\chi(\underline{\gamma})$.
\end{proof}

The proof of the preceding Proposition also shows that the action we have
defined is the unique with satisfies $\log\mathrm{hol}_{\underline{\Theta}%
}(\underline{\gamma})=\chi(\underline{\gamma})$. This is equivalent to
$\chi(\underline{\gamma})=%
{\textstyle\int\nolimits_{\gamma}}
\lambda-\alpha_{\phi}(x)$, and hence $\alpha_{\phi}(x)=%
{\textstyle\int\nolimits_{\gamma}}
\lambda-\chi(\underline{\gamma})$, that is our definition of $\alpha_{\phi
}(x)$.

For the elements in $\mathcal{G}_{0}$ (the connected component of the identity
in $\mathcal{G}$) we have a simpler result:

\begin{proposition}
\label{Connected}Let $\phi\in\mathcal{G}_{0}$ and $\varphi\subset\mathcal{G}$
be a curve such that $\varphi_{0}=1_{\mathcal{G}}$ and $\varphi_{1}=\phi$.
Then $\alpha_{\phi}(x)=\int\nolimits_{\varphi\cdot x}\lambda.$
\end{proposition}

\begin{proof}
The curve $\gamma=\varphi\cdot x$ is a curve joining $x$ and $\phi x$, and
$\underline{\gamma}=0$ on $Z_{1}(N/\mathcal{G}).$ Hence $\alpha_{\phi}%
(x)=\int\nolimits_{\gamma}\lambda-\chi(\underline{\gamma})=\int
\nolimits_{\varphi x}\lambda$.
\end{proof}

\begin{remark}
\emph{The preceding Proposition determines the action of }$\mathcal{G}_{0}%
$\emph{ only in terms of }$\lambda$\emph{, and without any reference to }%
$\chi$\emph{. Hence the differential character }$\chi$\emph{ is necessary only
to determine the action of the elements of }$\mathcal{G}$\emph{ not connected
with the identity. We note that Theorem \ref{Prop} is a generalization to non
connected groups of results in \cite{BCRS2}, \cite{Dupont} and \cite{Salamon}
for the space of connections.}
\end{remark}

\subsection{Chern-Simons differential characters\label{ChS}}

Chern-Simons theory allows to find, in a natural way, a differential character
with curvature a characteristic form. Let $G$ be a Lie group with a finite
number of connected components, $p\in I_{\mathbb{Z}}^{r}(G)$ and $q\colon
P\rightarrow N$ a principal $G$-bundle over a manifold $N$. If $A$ is a
connection on $q\colon P\rightarrow N$ with curvature $F$, we have
$p(F)\in\Omega^{k}(N)$. It can be seen (see \cite{Cheeger}) that if
$\Upsilon\in H^{2r}(\mathbf{B}G,\mathbb{Z})$\ is a universal characteristic
class compatible with $p$, there exist a differential character $\chi_{A}%
\in\hat{H}^{2r}(N)$ such that $\mathrm{curv}(\chi_{A})=p(F)$ and
$\mathrm{char}(\chi_{A})=\Upsilon_{P}$. We call $\chi_{A}$ the Chern-Simons
character of $p,\Upsilon$ and $A$. The Chern-Simons character is characterized
as being the unique natural map $(P,A)\mapsto\chi_{A}$ satisfying
$\mathrm{curv}(\chi_{A})=p(F)$ and $\mathrm{char}(\chi_{A})=\Upsilon_{P}$. We
recall that natural means that for any principal $G$-bundle $P^{\prime
}\rightarrow N^{\prime}$ and any $G$-bundle map $F\colon P^{\prime}\rightarrow
P$ we have $\chi_{F^{\ast}A}=f^{\ast}(\chi_{A})$, where $f\colon N^{\prime
}\rightarrow N$ is the map induced by $F$. Furthermore, if $A^{\prime}$ is
another connection on $P$, then we have
\begin{equation}
\chi_{A^{\prime}}=\chi_{A}+\varsigma(Tp(A^{\prime},A)). \label{AAprima}%
\end{equation}

A consequence of the preceding equation is the following (see e.g.
\cite[Proposition 2.9]{Cheeger})

\begin{lemma}
\label{LieCS}If $A_{t}$ is a smooth $1$-parametric family of connections on
$P$ with $\dot{A}_{0}=a\in\Omega^{1}(M,\mathrm{ad}P)$, then $\left.  \frac
{d}{dt}\right\vert _{t=0}\chi_{A_{t}}(u)=r\int_{u}p(a,F_{0},\overset
{(r-1)}{\ldots},F_{0})$ for every $u\in Z_{2r-1}(N)$.
\end{lemma}

\bigskip

\begin{remark}
\emph{The original Chern-Simons and Cheeger-Simons constructions are valid for
finite dimensional manifolds, but they can be extended to Banach or
Fr\'{e}chet infinite dimensional manifolds, and to more general types of
spaces (see for example \cite{BB}). Hence they can be applied to the infinite
dimensional spaces of connections and metrics.}
\end{remark}

\subsection{Fiber integration of differential characters}

\subsubsection{Integration on a product}

If $\alpha\in\Omega^{k}(C\times S)$ with $C$ compact and $\dim C=d$ we define
$\int\nolimits_{C}\alpha\in\Omega^{k-d}(S)$ by $\left(  \int\nolimits_{C}%
\alpha\right)  _{s}(X_{1},\ldots,X_{k-d})=\int\nolimits_{C}\iota_{X_{k-d}%
}\cdots\iota_{X_{1}}\alpha_{s}$ for $s\in S$, $X_{1},\ldots,X_{d}\in T_{s}S$.
If $k<d$ we define $\int\nolimits_{C}\alpha=0$. We have $\int\nolimits_{S}%
\int\nolimits_{C}\alpha=\int\nolimits_{C\times S}\alpha$, and also
$\int\nolimits_{C}\alpha=\int\nolimits_{C}\alpha^{d,k-d}$, where
$\alpha^{d,k-d}$ is the component relative to the bigraduation associated to
the product structure on $C\times S$. Furthermore we have Stokes theorem
$d\int\nolimits_{C}\alpha=\int\nolimits_{C}d\alpha-(-1)^{k-d}\int
\nolimits_{\partial C}\alpha$. If $c\colon C\rightarrow M$ is a map with $\dim
C=d$ we define maps $%
{\textstyle\int\nolimits_{c}}
\colon\Omega^{k}(M\times N)\rightarrow\Omega^{k-d}(N)$ by $%
{\textstyle\int\nolimits_{c}}
\alpha=%
{\textstyle\int\nolimits_{C}}
(c\times\mathrm{id}_{N})^{\ast}\alpha$ and we have $%
{\textstyle\int\nolimits_{c}}
\alpha=%
{\textstyle\int\nolimits_{c}}
\alpha^{d,k-d}$.

The integration map can be extended to differential characters in the
following way. If $\chi\in\hat{H}^{n}(M\times N)$ is a differential character
of order $n$ on $M\times N$ and $c\colon C\rightarrow M\ $is a smooth map with
$C$ closed, we define $\int_{c}\chi\in\hat{H}^{n-d}(N)$ by $(\int_{c}%
\chi)(s)=\chi(c\times s)$, and we have $\int_{c}\chi(\partial t)=\chi
(c\times\partial t)=%
{\textstyle\int\nolimits_{c\times t}}
\mathrm{curv}(\chi)=%
{\textstyle\int\nolimits_{t}}
{\textstyle\int\nolimits_{c}}
\mathrm{curv}(\chi)$. Hence $\int_{c}\chi$ is a differential character on
$N\mathcal{\ }$and its curvature is $\mathrm{curv}(\int_{c}\chi)=%
{\textstyle\int\nolimits_{c}}
\mathrm{curv}(\chi)$. Moreover, if $c=\partial u$ for some $u\colon
U\rightarrow M$ we have $%
{\textstyle\int\nolimits_{\partial u}}
\chi=\varsigma(%
{\textstyle\int\nolimits_{u}}
\mathrm{curv}(\chi))$ as $%
{\textstyle\int_{\partial u}}
\chi(s)=\chi(\partial u\times s)=\chi(\partial(u\times s))=%
{\textstyle\int\nolimits_{u\times s}}
\mathrm{curv}(\chi)=%
{\textstyle\int\nolimits_{s}}
{\textstyle\int\nolimits_{u}}
\mathrm{curv}(\chi).$

\subsubsection{Fiber integration}

The integration of differential characters can be extended to fiber
integration on a nontrivial bundle $\mathcal{N}\rightarrow N$ with fibre $M$
(e.g. see \cite{BB,Freed2}).\ In the product case we can integrate over any
submanifold of $M$, but for nontrivial bundles it only makes sense integration
over the fiber $M$ and over $\partial M$ if the fiber has boundary. If
$\mathcal{N}\rightarrow N$ is a fiber bundle with compact and oriented fibre
$M$ without boundary of dimension $d$, the fiber integration is a map
$\int_{M}\colon\hat{H}^{n}(\mathcal{N})\rightarrow\hat{H}^{n-d}(N)$ and
satisfies $\mathrm{curv}(%
{\textstyle\int\nolimits_{M}}
\chi)=%
{\textstyle\int\nolimits_{M}}
\mathrm{curv}(\chi)$, and $\mathrm{char}(%
{\textstyle\int\nolimits_{M}}
\chi)=%
{\textstyle\int\nolimits_{M}}
\mathrm{char}(\chi)$. If $M$ has boundary we have a map $\int_{\partial
M}\colon\hat{H}^{n}(\mathcal{N})\rightarrow\hat{H}^{n-d+1}(N)$ and
\begin{equation}%
{\textstyle\int\nolimits_{\partial M}}
\chi=\varsigma(%
{\textstyle\int\nolimits_{M}}
\mathrm{curv}(\chi)). \label{deltau}%
\end{equation}

Fiber integration satisfies the following naturality property (e.g. see
\cite{BB}): If $f\colon N^{\prime}\rightarrow N$ is a smooth map, $f^{\ast
}\mathcal{N}\rightarrow N^{\prime}$ is the pullback bundle and $\hat{f}\colon
f^{\ast}\mathcal{N}\rightarrow\mathcal{N}$ is the induced map, then we have
$\int_{M}\hat{f}^{\ast}\chi=f^{\ast}\left(  \int_{M}\chi\right)  $. If $M$ has
boundary we have $\int_{\partial M}\hat{f}^{\ast}\chi=f^{\ast}\left(
\int_{\partial M}\chi\right)  $.

If $F\colon\mathcal{N}^{\prime}\rightarrow\mathcal{N}$ is a morphism of
bundles with projection $f\colon N^{\prime}\rightarrow N$, then we have
$F=\hat{f}\circ\tilde{F}$ for a bundle morphism $\tilde{F}\colon
\mathcal{N}^{\prime}\rightarrow f^{\ast}\mathcal{N}$ with projects onto the
identity map. If $\tilde{F}$ is an isomorphism of bundles that preserves the
orientation on the fibers, then we have $\int_{M}\tilde{F}^{\ast}\chi=\int
_{M}\chi$ for any $\chi\in\hat{H}^{k}(f^{\ast}\mathcal{N})$. Hence we have the following

\begin{proposition}
\label{naturalityIntegral}Let $\mathcal{N}^{\prime}\rightarrow N^{\prime}$ and
$\mathcal{N}\rightarrow N$ be bundles with fiber $M$ and let $G\colon
\mathcal{N}^{\prime}\rightarrow\mathcal{N}$ be a morphism with projection
$f\colon N^{\prime}\rightarrow N$. If $\tilde{F}\colon\mathcal{N}^{\prime
}\rightarrow f^{\ast}\mathcal{N}$ is an isomorphism of bundles that preserves
the orientation on the fibers, then for any $\chi\in\hat{H}^{n}(\mathcal{N})$
we have $\int_{M}F^{\ast}\chi=f^{\ast}\left(  \int_{M}\chi\right)  $. If $M$
has boundary we have $\int_{\partial M}F^{\ast}\chi=f^{\ast}\left(
\int_{\partial M}\chi\right)  $.
\end{proposition}

\section{Equivariant deRham cohomology in the Cartan model\label{equi}}

We recall the definition of equivariant cohomology in the Cartan model
(\emph{e.g. }see \cite{BGV,GuiS}). Suppose that we have a left action of a
connected Lie group $\mathcal{G}$ on a manifold $N$. The map $X\mapsto
X_{N}(x)=\left.  \frac{d}{dt}\right\vert _{t=0}(\exp(-tX))(x)$ induces a\ Lie
algebra homomorphism $\mathrm{Lie\,}\mathcal{G}\rightarrow\mathfrak{X}(N)$.
The space of $\mathcal{G}$-equivariant differential forms is the space of
$\mathcal{G}$-invariant polynomials on $\mathrm{Lie\,}\mathcal{G}$ with values
in $\Omega^{\bullet}(N)$, $\Omega_{\mathcal{G}}(N)=\left(  \mathbf{S}%
^{\bullet}(\mathrm{Lie\,}\mathcal{G}^{\ast})\otimes\Omega^{\bullet}(N)\right)
^{\mathcal{G}}$ ($\mathcal{G}$\ acts on $\mathrm{Lie\,}\mathcal{G}$ by the
adjoint representation). The graduation on $\Omega_{\mathcal{G}}(N)$ is
defined by setting $\deg(\alpha)=2k+r$ if $\alpha\in\mathbf{S}^{k}%
(\mathrm{Lie\,}\mathcal{G}^{\ast})\otimes\Omega^{r}(N)$. Let $D\colon
\Omega_{\mathcal{G}}^{q}(N)\rightarrow\Omega_{\mathcal{G}}^{q+1}(N)$ be the
Cartan differential, $(D\alpha)(X)=d(\alpha(X))-\iota_{X_{N}}\alpha(X)$, for
$X\in\mathrm{Lie\,}\mathcal{G}.$ On $\Omega_{\mathcal{G}}^{\bullet}(N)$ we
have $D^{2}=0$, and the $\mathcal{G}$-equivariant cohomology (in the Cartan
model) of $N$ is defined as the cohomology of this complex.

A $\mathcal{G}$-equivariant $2$-form $\varpi$ is given by $\varpi
(X)=\sigma+\mu(X)$ where $\sigma$ is a $\mathcal{G}$-invariant $2$-form and
$\mu\colon\mathrm{Lie}\mathcal{G}\rightarrow\Omega^{0}(N)$ a linear
$\mathcal{G}$-equivariant map. The form $\varpi$ is $D$-closed if $d\sigma=0$
and $\iota_{X_{N}}\sigma=\mu(X)$ for every $X$. Hence $\mu$ is a co-moment map
for $\sigma$.

If a group acts on $M$, $N$, $C$ and $c\colon C\rightarrow M$ is $\mathcal{G}%
$-equivariant the integration map is extended to equivariant differential
forms $%
{\textstyle\int\nolimits_{c}}
\colon\Omega_{\mathcal{G}}^{k}(M\times N)\rightarrow\Omega_{\mathcal{G}}%
^{k-d}(N)$ by setting $\left(
{\textstyle\int\nolimits_{c}}
\alpha\right)  (X)=%
{\textstyle\int\nolimits_{c}}
(\alpha(X))$ for $X\in\mathrm{Lie\,}\mathcal{G}$, and we have $D\int
\nolimits_{C}\alpha=\int\nolimits_{C}D\alpha-(-1)^{k-d}\int\nolimits_{\partial
C}\alpha$.

\subsection{Equivariant characteristic classes in the Cartan
model\label{equiClass}}

We recall the definition of equivariant characteristic classes (see
\cite{BV1,BT} for details). Let {$\mathcal{G}$}$\ $be a group that acts (on
the left) on a principal $G$-bundle $\pi\colon P\rightarrow M${ }and let $A$
be a connection on $P$ invariant under the action of $\mathcal{G}$. It can be
proved (see \cite{BV1,BT}) that for every $X\in\mathrm{Lie}\mathcal{G}$ the{
$\mathfrak{g}$-valued function $A(X_{P})$ is of adjoint type and defines a
section of the adjoint bundle} $v_{A}(X)\in\Omega^{0}(N,\mathrm{ad}P)$. For
every $p\in I^{r}(G)$ the $\mathcal{G}$-equivariant characteristic form
$p_{\mathcal{G}}^{A}\in\Omega_{\mathcal{G}}^{2k}(N{)}$\ associated to $p$ and
$A$, is defined{ by} $p_{\mathcal{G}}^{A}(X)=p(F_{A}-v_{A}(X))$ for every
$X\in\mathrm{Lie}\mathcal{G}$.

A $\mathcal{G}$-equivariant $U(1)$-bundle is a principal $U(1)$-bundle
$\mathcal{U}\rightarrow N$ in which $\mathcal{G}$ acts by $U(1)$-bundle
automorphisms. If $\Xi\in\Omega^{1}(\mathcal{U},i\mathbb{R)}$ is a
$\mathcal{G}$-invariant connection then $\frac{i}{2\pi}D(\Xi)$ projects onto a
closed $\mathcal{G}$-equivariant $2$-form $\mathrm{curv}_{\mathcal{G}}(\Xi
)\in\Omega_{\mathcal{G}}^{2}(N))$ called the $\mathcal{G}$-equivariant
curvature of $\Xi$. If $X\in\mathrm{Lie}\mathcal{G}$ then $\mathrm{curv}%
_{\mathcal{G}}(\Xi)\!(X)\!=\mathrm{curv}(\Xi)-\frac{i}{2\pi}\iota
_{X_{\mathcal{U}}}\Xi$. If $\varpi\in\Omega_{\mathcal{G}}^{2}(N)$, a
$\mathcal{G}$-equivariant pre-quantization bundle for $\varpi$ is a principal
$U(1)$-bundle $\mathcal{U}\rightarrow N$ with a $\mathcal{G}$-invariant
connection $\Xi$ such that $\mathrm{curv}_{\mathcal{G}}(\Xi)=\varpi$.

\section{The space of connections\label{SectionConnections}}

Let $P\rightarrow M$ be a principal $G$-bundle, and $\mathcal{A}$ the space of
principal connections on this bundle. As $\mathcal{A}$ is an affine space
modeled on $\Omega^{1}(M,\mathrm{ad}P)$, we have canonical isomorphisms
$T_{A}\mathcal{A}\simeq\Omega^{1}(M,\mathrm{ad}P)$ for any $A\in\mathcal{A}$.
The Lie algebra of $\mathrm{Aut}P$ is the space of $G$-invariant vector fields
on $P$, $\mathrm{aut}P\subset\mathfrak{X}(P)$, and the Lie algebra of
$\mathrm{Gau}P$ is the subspace $\mathrm{gau}P$ of vertical $G$-invariant
vector fields. We have an identification $\mathrm{gau}P\simeq\Omega
^{0}(M,\mathrm{ad}P)$. The group $\mathrm{Aut}P$ acts on $\mathcal{A}$ and for
any $X\in\mathrm{aut}P$ we have $X_{\mathcal{A}}(A)=d^{A}(v_{A}(X))$. In
particular, if $X\in\mathrm{gau}P\simeq\Omega^{0}(M,\mathrm{ad}P)$ we have
$v_{A}(X)\simeq X$ and $X_{\mathcal{A}}(A)=d^{A}X$. The principal $G$-bundle
$\mathbb{P}=P\times\mathcal{A}\rightarrow M\times\mathcal{A}$ has a
tautological connection $\mathbb{A}\in\Omega^{1}(P\times\mathcal{A}%
,\mathfrak{g})$ defined by $\mathbb{A}_{(x,A)}(X,Y)=A_{x}(X)$ for $(x,A)\in
P\times\mathcal{A}$, $X$ $\in T_{x}P$, $Y\in T_{A}\mathcal{A}$. We denote by
$\mathbb{F}$ the curvature of $\mathbb{A}$ and we have $\mathbb{F}%
_{(x,A)}(a,a^{\prime})=0$, $\mathbb{F}_{(x,A)}(a,Y)=a(Y)$, $\mathbb{F}%
_{(x,A)}(Y,Y^{\prime})=F_{A}(Y,Y^{\prime})$ for $Y,Y^{\prime}\in T_{x}M$, and
$a,a^{\prime}\in T_{A}\mathcal{A}\simeq\Omega^{1}(M,\mathrm{ad}P)$. The group
$\mathrm{Aut}P$ acts on $\mathbb{P}$ by automorphisms and $\mathbb{A}$ is a
$\mathrm{Aut}P$-invariant connection. As the connection $\mathbb{A}$ is
$\mathrm{Aut}P$-invariant, for any $p\in I^{r}(G)$ we can define the
$\mathrm{Aut}P$-equivariant characteristic form $p_{\mathrm{Aut}P}%
^{\mathbb{A}}\in\Omega_{\mathrm{Aut}P}^{2r}(M\times\mathcal{A})$, given by
$p_{\mathrm{Aut}P}^{\mathbb{A}}(X)=p(\mathbb{F}-v_{\mathbb{A}}(X))$ for
$X\in\mathrm{aut}P$. If $M$ is a closed oriented manifold of dimension $n$ and
we consider the action of the group $\mathrm{Aut}^{+}P$, $p_{\mathrm{Aut}%
^{+}P}^{\mathbb{A}}\in\Omega_{\mathrm{Aut}^{+}P}^{2r}(M\times\mathcal{A})$ can
be integrated over $M$ to obtain $\int_{M}p_{\mathrm{Aut}^{+}P}^{\mathbb{A}%
}\in\Omega_{\mathrm{Aut}^{+}P}^{2r-n}(\mathcal{A})$. In particular, if
$n=2r-2$, we have $\varpi_{M}=\int_{M}p_{\mathrm{Aut}^{+}P}^{\mathbb{A}}%
\in\Omega_{\mathrm{Aut}^{+}P}^{2}(\mathcal{A})$ that can be written
$\varpi_{M}=\sigma_{M}+\mu_{M}$, with $\mu_{M}$ a co-moment map for
$\sigma_{M}$. If $M$ has boundary, we can integrate over $\partial M$ and we
obtain $\int_{\partial M}p_{\mathrm{Aut}^{+}P}^{\mathbb{A}}\in\Omega
_{\mathrm{Aut}^{+}P}^{2r-n+1}(\mathcal{A})$. In particular, if $n=2r-1$, we
have $\varpi_{\partial M}=\int_{\partial M}p_{\mathrm{Aut}^{+}P}^{\mathbb{A}%
}\in\Omega_{\mathrm{Aut}^{+}P}^{2}(\mathcal{A})$ that can be written
$\varpi_{\partial M}=\sigma_{\partial M}+\mu_{\partial M}$, with
$\mu_{\partial M}$ a co-moment map for $\sigma_{\partial M}$.

If we consider the action of the Gauge group, we have the $\mathrm{Gau}%
P$-equivariant characteristic form $p_{\mathrm{Gau}P}^{\mathbb{A}}\in
\Omega_{\mathrm{Gau}P}^{2r}(M\times\mathcal{A})$, given by $p_{\mathrm{Gau}%
P}^{\mathbb{A}}(X)=p(\mathbb{F}-X)$ for $X\in\mathrm{gau}P$. If $C$ is a
closed and oriented manifold of dimension $d$, for any map $c\colon
C\rightarrow M$ we can integrate $(c\times\mathrm{id}_{\mathcal{A}})^{\ast
}p_{\mathrm{Gau}P}^{\mathbb{A}}\in\Omega_{\mathrm{Gau}P}^{2r}(C\times
\mathcal{A})$ over $C$ to obtain $\int_{c}p_{\mathrm{Gau}P}^{\mathbb{A}}%
\in\Omega_{\mathrm{Gau}P}^{2r-d}(\mathcal{A})$. Again if $d=2r-2$, we have
$\varpi_{c}=\int_{c}p_{\mathrm{Gau}P}^{\mathbb{A}}\in\Omega_{\mathrm{Gau}%
P}^{2}(\mathcal{A})$ that can be written $\varpi_{c}=\sigma_{c}+\mu_{c}$, with
$\mu_{c}$ a co-moment map for $\sigma_{c}$.

The explicit expression of these forms is the following (see \cite{equiconn}).
For $A\in\mathcal{A}$, $a,b\in T_{A}\mathcal{A}\simeq\Omega^{1}(M,\mathrm{ad}%
P)$ and $X\in\mathrm{aut}P$ we have $(\sigma_{M})_{A}(a,b)=$\linebreak%
$r(r-1)\int_{M}p(a,b,F_{A},\overset{(r-2)}{\ldots},F_{A})$ and $(\mu_{M}%
)_{A}(X)=-r\int_{M}p(v_{A}(X),F_{A},\overset{(r-1)}{\ldots},F_{A})$, and
similar expressions for $\partial M$ and $c$.

As commented in the Introduction, our objective in this paper is to obtain
equivariant prequantization bundles of $(\mathcal{A},\varpi_{M})$,
$(\mathcal{A},\varpi_{\partial M})$ and $(\mathcal{A},\varpi_{c})$.

\section{Equivariant prequantization bundle\label{SectionPrequantization}}

In this section, we define the equivariant prequantization bundle. In Section
\ref{SectionDiscrete} we define the equivariant prequantization bundle for the
action of a discrete group. We show in Section \ref{SectionFirst} that the
definition for an arbitrary group can be reduced to the discrete case. In
place of working with the space of connections $\mathcal{A}$, we consider a
general connected and simply connected manifold $N$. It includes as particular
cases the space of connections for $N=\mathcal{A}$, and also the case of space
of Riemannian metrics.

We assume that $\mathcal{G}$ is a Lie group that acts (on the left) on the
following spaces

a) on a principal $G$-bundle $P\rightarrow M$ by\ $G$-bundle automorphisms,

b) on a connected and simply connected manifold $N$ and $\mathbb{A}$ is a
$\mathcal{G}$-invariant connection on the product bundle $\mathbb{P}=P\times
N\rightarrow M\times N$.

c) on a closed oriented manifold $C$ of dimension $d=2r-2$ and we have a
$\mathcal{G}$-equivariant map $c\colon C\rightarrow M$ and $\mathcal{G}$
preserves the orientation of $C$. We are interested in the following cases

c$_{1}$) The action of $\mathcal{G}$ on $M$ and $C$ is trivial (i.e.
$\mathcal{G}$ acts on $P$ by gauge transformations). In this case we can
consider any map $c\colon C\rightarrow M$.

c$_{2}$) $M$ is compact, $\partial M=0$ and $\mathcal{G}$ preserves the
orientation on $M$. In this case we can take $C=M$ and $c=\mathrm{id}_{M}$.

c$_{3}$) $M$ is an oriented manifold with compact boundary $\partial M$ and
$\mathcal{G}$ preserves the orientation on $M$. In this case we take
$C=\partial M$ and $c$ the inclusion $c\colon\partial M\hookrightarrow M$.

\bigskip

The following definitions are generalizations of the results in
\cite{equiconn} for connections and \cite{WP} for Riemannian metrics.

As the connection $\mathbb{A}$ is $\mathcal{G}$-invariant, for any polynomial
$p\in I^{r}(G)$ we can define the $\mathcal{G}$-equivariant characteristic
class $p_{\mathcal{G}}^{\mathbb{A}}\in\Omega_{\mathcal{G}}^{2r}(M\times N)$,
given by $p_{\mathcal{G}}^{\mathbb{A}}(X)=p(\mathbb{F}-v_{\mathbb{A}}(X))$ for
$X\in\mathrm{Lie}\mathcal{G}$. As $c\colon C\rightarrow M$ is $\mathcal{G}%
$-invariant we can integrate $(c\times\mathrm{id}_{M})^{\ast}p_{\mathcal{G}%
}^{\mathbb{A}}\in\Omega_{\mathcal{G}}^{2r}(C\times N)$ over $C$ to obtain
$\varpi_{c}=\int_{c}p_{\mathcal{G}}^{\mathbb{A}}\in\Omega_{\mathcal{G}}%
^{2}(N)$.\ We have $\varpi_{c}=\sigma_{c}+\mu_{c}$, with $\sigma_{c}=\int
_{c}p(\mathbb{F})$ and $\mu_{c}$ a co-moment map for $\sigma_{c}$ given by
$\mu_{c}(X)=-r\int_{c}p(v_{\mathbb{A}}(X),\mathbb{F}\mathbf{,}\overset
{(r-1)}{\mathbf{\ldots}}\mathbf{,}\mathbb{F})$ for $X\in\mathrm{Lie}%
\mathcal{G}$.

As it is commented in the Introduction, the equivariant prequantization
bundles are given in terms of\ a background connection). Let $A_{0}$ be a
connection on $P\rightarrow M$ (we call $A_{0}$ a background connection. If
$\mathrm{pr}_{1}\colon P\times N\rightarrow P$ denotes the projection, then
$\mathbb{A}$ and $\overline{A}_{0}=\mathrm{pr}_{1}^{\ast}A_{0}$ are
connections on the same bundle $P\times N\rightarrow M\times N$, and hence we
can define $Tp(\mathbb{A},\overline{A}_{0})\in\Omega^{2r-1}(M\times N)$. The
product structure on $M\times N$ induces a bigraduation $\Omega^{k}(M\times
N)\simeq\bigoplus\nolimits_{i=0}^{k}\Omega^{i,k-i}(M\times N)$. We have
$p(\mathbb{F})=dTp(\mathbb{A},\overline{A}_{0})+\mathrm{pr}_{1}^{\ast}%
p(F_{0})$, with $\mathrm{pr}_{1}^{\ast}p(F_{0})\in\Omega^{2r,0}(M\times N)$.
Hence for any $u\colon U\rightarrow M$ with $d=\dim U<2r-1$ we have%
\begin{equation}%
{\textstyle\int_{u}}
p(\mathbb{F})\!=\!%
{\textstyle\int_{u}}
\!dTp(\mathbb{A},\overline{A}_{0})\!=\!d%
{\textstyle\int_{u}}
\!Tp(\mathbb{A},\overline{A}_{0})\!-\!(-1)^{d}\!%
{\textstyle\int_{\partial u}}
\!Tp(\mathbb{A},\overline{A}_{0}), \label{backgr1}%
\end{equation}
where we have used that $\int_{u}\mathrm{pr}_{1}^{\ast}p(F_{0})=0$ as
$\mathrm{pr}_{1}^{\ast}p(F_{0})\in\Omega^{2r,0}(M\times N)$. In particular, if
we define $\rho_{c}=\int_{c}Tp(\mathbb{A},\overline{A}_{0})\in\Omega^{1}(N)$,
then by using equation (\ref{backgr1}) and that $\partial c=0$ we obtain
\begin{equation}
\sigma_{c}=%
{\textstyle\int_{c}}
p(\mathbb{F})=d\rho_{c}. \label{dRho}%
\end{equation}

\subsection{Discrete group\label{SectionDiscrete}}

Assume that $\mathcal{G}$ is a discrete group. Let $E$ be\ a manifold in which
$\mathcal{G}$ acts and such that the following condition is satisfied:

(*) $E$ is connected and simply connected and $\pi\colon N\times
E\rightarrow(N\times E)/\mathcal{G}$ is a principal $\mathcal{G}$-bundle.

For example we can take $E=\mathbf{E}\mathcal{G}$ or another simpler manifold.
We denote by $q\colon M\times N^{\prime}\rightarrow M\times N$ and
$\overline{q}\colon\mathbb{P}^{\prime}=\mathbb{P}\times E\rightarrow
\mathbb{P}$ the projections, which are $\mathcal{G}$-equivariant maps. Hence
$\overline{q}^{\ast}\mathbb{A}$ is a $\mathcal{G}$-invariant connection on
$\mathbb{P}^{\prime}\rightarrow M\times N^{\prime}$, and it projects onto a
connection $\underline{\mathbb{A}}$ on the quotient principal $G$-bundle
$\mathbb{P}^{\prime}/\mathcal{G}\rightarrow(M\times N^{\prime})/\mathcal{G}$.
We denote by $\underline{\mathbb{F}}$ the curvature of $\underline{\mathbb{A}%
}$. Given $(p,\Upsilon)\in\mathcal{I}_{\mathbb{Z}}^{r}(G)$ we have the
Chern-Simons character $\chi_{\underline{\mathbb{A}}}\in\hat{H}^{2r}((M\times
N^{\prime})/\mathcal{G})$. As $c\colon C\rightarrow M$ is a $\mathcal{G}%
$-equivariant map, it induces a map $\underline{c\times\mathrm{id}_{N}}%
\colon(C\times N^{\prime})/\mathcal{G}\rightarrow(M\times N^{\prime
})/\mathcal{G}$. The character $(\underline{c\times\mathrm{id}_{N^{\prime}}%
})^{\ast}\chi_{\underline{\mathbb{A}}}\in\hat{H}^{2r}((C\times N^{\prime
})/\mathcal{G})$ can be integrated over the fibre of $(C\times N^{\prime
})/\mathcal{G}\rightarrow N^{\prime}/\mathcal{G}$ and we obtain a differential
character $\xi_{c}=\int_{c}\chi_{\underline{\mathbb{A}}}=\int_{C}%
(\underline{c\times\mathrm{id}_{N}})^{\ast}\chi_{\underline{\mathbb{A}}}%
\in\hat{H}^{2}(N^{\prime}/\mathcal{G})$. We have $\mathrm{curv}(\xi_{c}%
)=\int_{c}\mathrm{curv}(\chi_{\underline{\mathbb{A}}})=\int_{c}p(\underline
{\mathbb{F}})$ and $\mathrm{char}(\xi_{c})=\int_{c}\mathrm{char}%
(\chi_{\underline{\mathbb{A}}})=\int_{c}\Upsilon_{\mathbb{P}^{\prime
}/\mathcal{G}}$.

If $A_{0}$ is a background connection,\ by equation (\ref{dRho}) we have
$\pi^{\ast}(\int_{c}p(\underline{\mathbb{F}}))=\int_{c}p(q^{\ast}%
\mathbb{F})=q^{\ast}\int_{c}p(\mathbb{F})=d(q^{\ast}\rho_{c})$. By applying
Proposition \ref{Prop} with $\lambda=q^{\ast}\rho_{c}$we obtain a cocycle
$\bar{\alpha}_{c}\colon\mathcal{G}\times N\times E\rightarrow\mathbb{R}%
/\mathbb{Z}$. Precisely, if $\phi\in\mathcal{G}$, $\gamma$ is a curve on $N$
joining $x$ and $\phi x$ and $\gamma^{\prime}$ is a curve on $E$ joining $e$
and $\phi e$ we have
\[
\bar{\alpha}_{\phi}(x,e)=\int\nolimits_{\gamma\times\gamma^{\prime}}q^{\ast
}\rho_{c}-\xi_{c}(\pi\circ(\gamma\times\gamma^{\prime}))=\int\nolimits_{\gamma
}\rho_{c}-\xi_{c}(\pi\circ(\gamma\times\gamma^{\prime})).
\]

\begin{remark}
\emph{Note that by Lemma \ref{cambiopunto} }$\bar{\alpha}_{\phi}$\emph{ is
differentiable and }$d\bar{\alpha}_{\phi}=\phi^{\ast}q^{\ast}\rho_{c}-q^{\ast
}\rho_{c}=q^{\ast}(\phi^{\ast}\rho_{c}-\rho_{c})$.
\end{remark}

\begin{lemma}
\label{independenciaEdiscreto}We have

a) $\bar{\alpha}_{\phi}(x,e)$ does not depend on $e\in E$, and hence
$\bar{\alpha}=q^{\ast}\alpha$ for a cocycle $\alpha\colon\mathcal{G}\times
N\rightarrow\mathbb{R}/\mathbb{Z}$. Furthermore, $\alpha$ satisfies
$\alpha_{\phi_{2}\phi_{1}}(x)=\alpha_{\phi_{1}}(x)+\alpha_{\phi_{2}}(\phi
_{1}x)$ and $d\alpha_{\phi}=\phi^{\ast}\rho_{c}-\rho_{c}$ for $\phi$,
$\phi_{1}$, $\phi_{2}\in\mathcal{G}$.

b) $\alpha_{c}$ does not dependent on the manifold $E$ chosen.
\end{lemma}

\begin{proof}
a) By Lemma \ref{cambiopunto}, if $e^{\prime}$ is another point on $E$ and
$\gamma_{ee^{\prime}}$ a curve on $E$ joining $e$ and $e^{\prime}$ (it exists
as $E$ is connected) we have $\alpha_{\phi}(x,e^{\prime})=\alpha_{\phi
}(x,e)+\int\nolimits_{\{x\}\times\gamma_{ee^{\prime}}}(\phi^{\ast}q^{\ast}%
\rho_{c}-q^{\ast}\rho_{c})$. But $\int\nolimits_{\{x\}\times\gamma
_{ee^{\prime}}}(\phi^{\ast}q^{\ast}\rho_{c}-q^{\ast}\rho_{c})=\int
\nolimits_{\{x\}\times\gamma_{ee^{\prime}}}q^{\ast}(\phi^{\ast}\rho_{c}%
-\rho_{c})=0$, and hence $\alpha_{\phi}(x,e^{\prime})=\alpha_{\phi}(x,e)$.

b) Let $E_{1}$, $E_{2}$ be two manifolds satisfying condition (*). Then
$E_{3}=E_{1}\times E_{2}$ also satisfies (*). We define $Q_{i}\colon(P\times
N\times E_{3})/\mathcal{G}\rightarrow(P\times N\times E_{i})/\mathcal{G}$, and
$q_{i}\colon(N\times E)/\mathcal{G}\rightarrow(N\times E_{i})/\mathcal{G}$,
$i=1,2$. We have $Q_{1}^{\ast}\underline{\overline{q}_{1}^{\ast}\mathbb{A}%
}\mathbb{=}Q_{2}^{\ast}\underline{\overline{q}_{2}^{\ast}\mathbb{A}%
}=\underline{\overline{q}_{3}^{\ast}\mathbb{A}}$, and by using Proposition
\ref{naturalityIntegral}\ we obtain $q_{1}^{\ast}\xi_{c}^{1}\mathbb{=}%
q_{2}^{\ast}\xi_{c}^{2}=\xi_{c}^{3}$.

If $\phi\in\mathcal{G}$, $\gamma$ is a curve on $N$ joining $x$ and $\phi x$,
and $\gamma_{i}$ is a curve on $E_{i}$ joining $e_{i}$ and $\phi e_{i}$ we
have $\xi_{c}^{3}(\pi_{3}\circ(\gamma\times\gamma_{1}\times\gamma_{2}%
))=\xi_{c}^{i}(q_{i}\circ\pi_{3}\circ(\gamma\times\gamma_{1}\times\gamma
_{2}))=\xi_{c}^{i}(\pi_{i}\circ(\gamma\times\gamma_{i}))$ for $i=1,2$, and by
the definition of $\overline{\alpha}_{i}$ and a) we have $(\alpha_{1})_{\phi
}(x)=(\overline{\alpha}_{1})_{\phi}(x,e_{1})=(\overline{\alpha}_{3})_{\phi
}(x,e_{1},e_{2})=(\overline{\alpha}_{2})_{\phi}(x,e_{2})=(\alpha_{2})_{\phi
}(x)$.
\end{proof}

The cocycle $\alpha\colon\mathcal{G}\times N\rightarrow\mathbb{R}/\mathbb{Z}$
defines an action of $\mathcal{G}$ on $\mathcal{U}_{c}=N\times U(1)$ by
$U(1)$-bundle automorphisms $\Phi_{\phi}(x,u)=(\phi x,\exp(2\mathrm{\pi
}i\alpha_{\phi}(x))\cdot u)\ $and the connection form $\Xi_{c}=\theta-2\pi
i\rho_{c}$ is $\mathcal{G}$-invariant.

Hence, for any action of a discrete group $\mathcal{G}$ on $N$ we have the following

\begin{proposition}
\label{PropIntCSDiscreto copy(1)} Let $A_{0}$ be a background connection on
$P\rightarrow M$. Then there exists a lift of the action of $\mathcal{G}$ on
$N$ to an action on $\mathcal{U}_{c}=N\times U(1)$ by $U(1)$-bundle
automorphisms such that $\Xi_{c}=\theta-2\pi i\rho_{c}\in\Omega^{1}(N\times
U(1),i\mathbb{R})$ is $\mathcal{G}$-invariant.
\end{proposition}

We recall that a $\mathcal{G}$-equivariant section of $\mathcal{U}%
_{c}\rightarrow N$ is determined by a map $S\colon N\rightarrow U(1)$
$S(x)=\exp(2\pi i\cdot s(x))$ where $s\colon N\rightarrow\mathbb{R}%
/\mathbb{Z}$ satisfies $\alpha_{\phi}(x)=s(\phi x)-s(x)$. The following result
shows that our bundle generalizes the Chern-Simons line

\begin{proposition}
If $c=\partial u$ for a $\mathcal{G}$-equivariant map $u\colon U\rightarrow M$
and we define $s_{u}=-%
{\textstyle\int\nolimits_{u}}
Tp(\mathbb{A},\overline{A}_{0})\in\Omega^{0}(N)$, then $\alpha_{\phi}%
(x)=s_{u}(\phi x)-s_{u}(x)$. Hence $S_{u}=\exp(2\mathrm{\pi}i\cdot s_{u})$
determines a $\mathcal{G}$-equivariant section of $\mathcal{U}_{c}\rightarrow
N$.

Furthermore, we have $\nabla^{\Xi_{c}}S_{u}=-2\mathrm{\pi}i\sigma_{u}\cdot
S_{u}$, where $\sigma_{u}=\int_{u}p(\mathbb{F})$.
\end{proposition}

\begin{proof}
If we define $\sigma_{u}=\int_{u}p(\mathbb{F})$ then $d\sigma_{u}=\int
_{u}d(p(\mathbb{F}))+\int_{\partial u}p(\mathbb{F})=\sigma_{c}$. Moreover by
equations (\ref{backgr1}) and (\ref{dRho}) we have $\sigma_{u}=\int
_{u}p(\mathbb{F})=d\int_{u}Tp(\mathbb{A},\overline{A}_{0})+\int
\nolimits_{\partial u}Tp(\mathbb{A},\overline{A}_{0})=-ds_{u}+\rho_{c}.$ For
$\gamma$ a curve joining $x$ and $\phi x$, and $\gamma^{\prime}$ a curve on
$E$ joining $e$ and $\phi e$, using the preceding equations and equation
(\ref{deltau}) we have%
\begin{multline*}
\alpha_{\phi}(x)=%
{\textstyle\int\nolimits_{\gamma}}
\rho_{c}-\left(
{\textstyle\int_{\partial u}}
\chi_{\underline{\mathbb{A}}}\right)  (\pi\circ(\gamma\times\gamma^{\prime}))=%
{\textstyle\int\nolimits_{\gamma}}
(\sigma_{u}+ds_{u})-%
{\textstyle\int\nolimits_{\gamma\times\gamma^{\prime}}}
q^{\ast}%
{\textstyle\int\nolimits_{u}}
p(\mathbb{F})\\
=%
{\textstyle\int\nolimits_{\gamma}}
(\sigma_{u}+ds_{u})-%
{\textstyle\int\nolimits_{\gamma}}
\sigma_{u}=%
{\textstyle\int\nolimits_{\gamma}}
ds_{u}=%
{\textstyle\int\nolimits_{\partial\gamma}}
s_{u}=s_{u}(\phi x)-s_{u}(x).
\end{multline*}

Finally we have $\nabla^{\Xi_{c}}S_{u}=dS_{u}-2\mathrm{\pi}i\rho_{c}\cdot
S_{u}=2\mathrm{\pi}ids_{u}\cdot S_{u}-2\mathrm{\pi}i\rho_{c}\cdot
S_{u}=-2\mathrm{\pi}i\sigma_{u}\cdot S_{u}$.
\end{proof}

\begin{remark}
\emph{We can also consider that }$\exp(2\pi i\cdot%
{\textstyle\int\nolimits_{u}}
Tp(\mathbb{A},\overline{A}_{0}))$\emph{ determines a section of the inverse
bundle} $\mathcal{U}_{c}^{-1}$ \emph{as it is done in \cite{Dupont}.}
\end{remark}

Let $\mathcal{H}$ be another discrete group and let $h\colon\mathcal{H}%
\rightarrow\mathcal{G}$ be a group homomorphism. $h$ induces actions\ of the
group $\mathcal{H}$ on $N$ and $P$ that satisfies conditions a), b) and c). If
$\phi\in\mathcal{H}$ we denote by $\alpha_{\phi}^{\mathcal{H}}$ and
$\Phi_{\phi}^{\mathcal{H}}$ the cocycle and action determined by the group
$\mathcal{H}$

\begin{proposition}
\label{subgroupDiscreto}We have $\alpha_{\phi}^{\mathcal{H}}=\alpha_{h(\phi
)}^{\mathcal{G}}$ and $\Phi_{\phi}^{\mathcal{H}}=\Phi_{h(\phi)}^{\mathcal{G}}$
for any $\phi\in\mathcal{H}$.
\end{proposition}

\begin{proof}
We define $N_{1}=N\times E\mathcal{G}$, $N_{2}=N\times E\mathcal{G}\times
E\mathcal{H}$, and the projections $\overline{q}_{1}\colon P\times
N_{1}\rightarrow P\times N$, $\overline{q}_{2}\colon P\times N_{2}\rightarrow
P\times N$. $h$ induces actions of $\mathcal{H}$ on $P$, $N$ and
$E\mathcal{G}$. We denote by $\underline{\mathbb{A}}_{\mathcal{G}}\in
\Omega^{1}((P\times N_{1})/\mathcal{G},\mathfrak{g})$ and $\underline
{\mathbb{A}}_{\mathcal{H}}\in\Omega^{1}((P\times N_{2})/\mathcal{H}%
,\mathfrak{g})$ the projections of the connections $\overline{q}_{1}^{\ast
}\mathbb{A}$ and $\overline{q}_{2}^{\ast}\mathbb{A}$, and by $\xi
_{c}^{\mathcal{H}}=\int_{c}\chi_{\underline{\mathbb{A}}_{\mathcal{H}}}\in
\hat{H}^{2}(N_{2}/\mathcal{H})$ and $\xi_{c}^{\mathcal{G}}=\int_{c}%
\chi_{\underline{\mathbb{A}}_{\mathcal{G}}}\in\hat{H}^{2}(N_{1}/\mathcal{G})$
the integrated Chern-Simons characters.

We define a map $\overline{Z}\colon(P\times N_{2}\mathbf{)}/\mathcal{H}%
\rightarrow(P\times N_{1})/\mathcal{G}$ by $\overline{Z}([y,x,e_{1}%
,e_{2}]_{\mathcal{H}})=[y,x,e_{1}]_{\mathcal{G}}$ and we have $\underline
{\mathbb{A}}_{\mathcal{H}}=\overline{Z}^{\ast}(\underline{\mathbb{A}%
}_{\mathcal{G}})$. In a similar way we define the maps $Z\colon(M\times
N_{2}\mathbf{)}/\mathcal{H}\rightarrow(M\times N_{1})/\mathcal{G}$ and
$z\colon N_{2}/\mathcal{H}\rightarrow N_{1}/\mathcal{G}$. $Z$ determines a
morphism of bundles of fiber $M$ satisfying the conditions of Proposition
\ref{naturalityIntegral} and we conclude that $\xi_{c}^{\mathcal{H}}=z^{\ast
}\xi_{c}^{\mathcal{G}}$.

Let $\gamma$ be a curve on $N$ joining $x$ and $\phi\cdot x=h(\phi)\cdot x$,
$\gamma_{1}$a curve on $E\mathcal{G}$ joining $e_{1}$ and $h(\phi)\cdot e_{1}$
and $\gamma_{2}$ a curve on $E\mathcal{H}$ joining $e_{2}$ and $\phi\cdot
e_{2}$, and let $\vec{\gamma}_{1}=\gamma\times\gamma_{1}$. $\vec{\gamma}%
_{2}=\gamma\times\gamma_{1}\times\gamma_{2}$. By Lemma
\ref{independenciaEdiscreto} we have $\alpha_{\phi}^{\mathcal{H}}%
(x)=\int\nolimits_{\gamma}\rho_{c}-\xi_{c}^{\mathcal{H}}(\pi_{\mathcal{H}%
}\circ\vec{\gamma}_{2})$ and $\alpha_{h(\phi)}^{\mathcal{G}}(x)=\int
\nolimits_{\gamma}\rho_{c}-\xi_{c}^{\mathcal{G}}(\pi_{\mathcal{G}}\circ
\vec{\gamma}_{1})$, where $\pi_{\mathcal{G}}\colon N_{1}\rightarrow
N_{1}/\mathcal{G}$, and $\pi_{\mathcal{H}}\colon N_{2}\rightarrow
N_{2}/\mathcal{H}$ are the projections. We have $\xi_{c}^{\mathcal{H}}%
(\pi_{\mathcal{H}}\circ\vec{\gamma}_{2})=z^{\ast}\xi_{c}^{\mathcal{G}}%
(\pi_{\mathcal{H}}\circ\vec{\gamma}_{2})=\xi_{c}^{\mathcal{G}}(z\circ
\pi_{\mathcal{H}}\circ\vec{\gamma}_{2})=\xi_{c}^{\mathcal{G}}(\pi
_{\mathcal{G}}\circ\vec{\gamma}_{1})$, and hence $\alpha_{\phi}^{\mathcal{H}%
}(x)=\alpha_{h(\phi)}^{\mathcal{G}}(x)$.
\end{proof}

\subsection{General Group\label{SectionFirst}}

In this Section we give a definition of the prequantization bundle valid for
arbitrary Lie groups. The definition for discrete groups can not be
generalized directly to Lie groups because the connection $\overline{q}^{\ast
}\mathbb{A}$ is not necessarily projectable to the quotient $(P\times N\times
E)/\mathcal{G}$. As commented in the Introduction, it is possible to obtain a
connection on the quotient space by using an auxiliary connection
$\mathfrak{A}$ on the principal $\mathcal{G}$-bundle $N\times E\rightarrow
(N\times E)/\mathcal{G}$. As this could be problematic for Fr\'{e}chet Lie
groups, we define the lift of the action of any element $\phi\in\mathcal{G}$
by using the results for discrete groups.

Given $\phi\in\mathcal{G}$, the homomorphism $\mathbb{Z}\rightarrow
\mathcal{G}$ $n\mapsto\phi^{n}$ determines actions of the group$\ \mathbb{Z}$
on $P$ and $N$. We apply the results of Section \ref{SectionDiscrete}\ to the
discrete group $\mathbb{Z}$ and we obtain a cocycle $\alpha_{\phi}\colon
N\rightarrow\mathbb{R}/\mathbb{Z}$ and a lifted action $\Phi_{\phi}\colon
N\times U(1)\rightarrow N\times U(1)$ by $U(1)$-bundle automorphisms that
leaves invariant\ the connection $\Xi_{c}=\theta-2\pi i\rho_{c}$. Let us
compute $\alpha_{\phi}$ explicitly. We consider the universal $\mathbb{Z}%
$-bundle $\mathbf{E}\mathbb{Z=R\rightarrow Z}$, and the products $\overline
{q}\colon P\times N\times\mathbb{R\rightarrow}P\times N$, $q\colon M\times
N\times\mathbb{R\rightarrow}M\times N$. The connection $\overline{q}^{\ast
}\mathbb{A}$ is $\mathbb{Z}$-invariant and hence projects onto a connection
$\underline{\mathbb{A}}_{\phi}$ on the principal $G$-bundle $\mathbb{(}P\times
N\times\mathbb{R)}/\mathbb{Z}\rightarrow\mathbb{(}M\times N\times
\mathbb{R)}/\mathbb{Z}$. Given $(p,\Upsilon)\in\mathcal{I}_{\mathbb{Z}}(G)$,
we have the Chern-Simons differential character $\chi_{\underline{\mathbb{A}%
}_{\phi}}\in\hat{H}^{2r}(\mathbb{(}M\times N\times\mathbb{R)}/\mathbb{Z})$,
and by integrating over $c$ we obtain the integrated Chern-Simons character
$\xi_{c}^{\phi}=\int_{c}\chi_{\underline{\mathbb{A}}_{\phi}}\in\hat{H}%
^{2}((N\times\mathbb{R)}/\mathbb{Z})$. If $\gamma$ is any curve on $N$ joining
$x$ and $\phi x$, we define $\vec{\gamma}\colon I\rightarrow N\times
\mathbb{R}$, $\vec{\gamma}(s)=(\gamma(s),s)$ and we have%
\[
\alpha_{\phi}(x)=\int\nolimits_{\vec{\gamma}}q^{\ast}\rho_{c}-\xi_{c}^{\phi
}(\pi_{\phi}\circ\vec{\gamma})=\int\nolimits_{\gamma}\rho_{c}-\xi_{c}^{\phi
}(\pi_{\phi}\circ\vec{\gamma})\text{,}%
\]
where $\pi_{\phi}\colon N\times\mathbb{R\rightarrow(}N\times\mathbb{R)}%
/\mathbb{Z}$ denotes the projection.

Note that $\mathbb{(}N\times\mathbb{R)}/\mathbb{Z}$ can be identified with
$\mathbb{(}N\times I\mathbb{)}/\mathbb{\sim}_{\phi}$ where the equivalence
relation is defined by $(x,0)\mathbb{\sim}_{\phi}(\phi x,1)$.With this
identification we have $\pi_{\phi}\circ\vec{\gamma}(s)=[\gamma(s),s]_{\phi}$,
which is a closed curve as we have $(\gamma(1),1)=(\phi x,1)\mathbb{\sim
}_{\phi}(x,0)=(\gamma(0),0)$.

The map $\Phi_{\phi}\colon N\times U(1)\rightarrow N\times U(1),$ $\Phi_{\phi
}(x,u)=(\phi x,\exp(2\mathrm{\pi}i\alpha_{\phi}(x))\cdot u)$ defines a group
action of $\mathcal{G}$ on $N\times U(1)$ as we have the following

\begin{proposition}
\label{GroupAction}For any $\phi_{1}$, $\phi_{2}\in\mathcal{G}$ we have
$\alpha_{\phi_{2}\phi_{1}}(x)=\alpha_{\phi_{1}}(x)+\alpha_{\phi_{2}}(\phi
_{1}x)$, and hence $\Phi_{\phi_{2}\phi_{1}}=\Phi_{\phi_{2}}\circ\Phi_{\phi
_{1}}$.
\end{proposition}

\begin{proof}
Given $\phi_{1}$, $\phi_{2}\in\mathcal{G}$ we consider the free group
$F_{2}[x_{1},x_{2}]$ generated by two elements $x_{1},x_{2}$. The assignment
$x_{1}\mapsto\phi_{1}$, $x_{2}\mapsto\phi_{2}$ defines and action of the
discrete group $F_{2}[x_{1},x_{2}]$ on $P\times N$. We have two homomorphisms
$h_{i}\colon\mathbb{Z}\rightarrow F_{2}[x_{1},x_{2}]$ determined by setting
$h_{i}(n)=x_{i}^{n}$. By applying Lemma \ref{independenciaEdiscreto} and
Proposition \ref{subgroupDiscreto} we obtain $\alpha_{\phi_{2}\phi_{1}%
}(x)=\alpha_{\phi_{1}}(x)+\alpha_{\phi_{2}}(\phi_{1}x)$ for any $x\in N$.
\end{proof}

\begin{lemma}
\label{PropContinuo}If $\phi_{t}\in\mathcal{G}$, $t\in(-\varepsilon
,\varepsilon)$ is a local $1$-parametric subgroup of $\mathcal{G}$ with
$\dot{\phi}_{t}=X\in\mathrm{Lie}\mathcal{G}$, then we have $\left.
\frac{d\alpha_{\phi_{t}}}{dt}\right\vert _{t=0}=\rho_{c}(X_{N})+\mu_{c}(X)$.
\end{lemma}

\begin{proof}
Given $x\in N$, we define $\gamma(s)=\phi_{s}x$ and for $t\in(-\varepsilon
,\varepsilon)$, $\gamma^{t}(s)=\gamma(ts)$ and $\vec{\gamma}^{t}%
(s)=(\gamma^{t}(s),s)$. We have $\vec{\gamma}^{t}(0)=(\phi_{0}x,0)=(x,0)$ and
$\vec{\gamma}^{t}(1)=(\phi_{t}x,1)\sim_{\phi_{t}}(x,0)$. By definition we have
$\alpha_{\phi_{t}}(x)=\int\nolimits_{\gamma^{t}}\rho_{c}-\xi_{c}^{\phi_{t}%
}(\pi_{\phi_{t}}\circ\vec{\gamma}^{t}).$

The derivative of the first term is easy to compute, as we have $\int
\nolimits_{\gamma^{t}}\rho_{c}=\int\nolimits_{0}^{1}\rho_{c}(\dot{\gamma}%
^{t}(s))ds=\int\nolimits_{0}^{1}\rho_{c}(\dot{\gamma}(ts))tds\underset
{z=ts}{=}\int\nolimits_{0}^{t}\rho_{c}(\dot{\gamma}(z))dz$ and hence $\frac
{d}{dt}\left.  \int\nolimits_{\gamma^{t}}\rho_{c}\right\vert _{t=0}=\rho
_{c}(X_{N}(x))$.

Next we compute the derivative of the second term. We denote by $N_{t}$ the
manifold $N\times\mathbb{R}$ with the action of $\mathbb{Z}$ determined by
$n\cdot(x,s)=((\phi_{t})^{n}x,s+n)$. The maps $w_{t}\colon N_{0}\rightarrow
N_{t}$, $w_{t}(x,s)=(\phi_{ts}x,s)$, $W_{t}\colon P\times N_{0}\rightarrow
P\times N_{t}$, $W_{t}(y,x,s)=(\phi_{ts}y,\phi_{ts}x,s)$ for $y\in P,x\in N$
and $s\in\mathbb{R}$ are $\mathbb{Z}$-equivariant. We denote by $\underline
{w}_{t}\colon N_{0}/\mathbb{Z}\rightarrow N_{t}/\mathbb{Z}$ and $\underline
{W}_{t}\colon(P\times N_{0})/\mathbb{Z}\rightarrow(P\times N_{t})/\mathbb{Z}$
the induced maps. We define the connections $A_{t}=W_{t}^{\ast}q^{\ast
}\mathbb{A}$, and we have $A_{0}=q^{\ast}\mathbb{A}$ and $\dot{A}_{0}=\left.
\frac{d}{dt}\right\vert _{t=0}(W_{t}^{\ast}q^{\ast}\mathbb{A)}=L_{Y}q^{\ast
}\mathbb{A}$, where $Y=\left.  \frac{dW_{t}}{dt}\right\vert _{t=0}=sX_{P\times
N}$. As $q^{\ast}\mathbb{A}$ is $\mathcal{G}$-invariant we have $\dot{A}%
_{0}=L_{Y}q^{\ast}\mathbb{A}=sq^{\ast}(L_{X_{P\times N}}\mathbb{A})+q^{\ast
}(\iota_{X_{P\times N}}\mathbb{A)}ds=(q^{\ast}v_{\mathbb{A}}(X))ds$.

We denote by $\underline{A}_{t}$ the projection of $A_{t}$ on $(P\times
N_{0})/\mathbb{Z}$ and we define $\zeta_{c}^{t}=\int_{c}\chi_{\underline
{A}_{t}}\in\hat{H}^{2}(N_{0}/\mathbb{Z})$. In each of the cases c$_{1}$),
c$_{2}$) and c$_{3}$) by using Proposition \ref{naturalityIntegral} we obtain
$\zeta_{c}^{t}=\underline{w}_{t}^{\ast}(\xi_{c}^{\phi_{t}})$.

If $\varrho_{x}\colon I\rightarrow N_{0}$ is the curve $\varrho_{x}(s)=(x,s)$
then we have$\ \pi_{\phi_{t}}\circ\vec{\gamma}^{t}=\underline{w}_{t}\circ
\pi_{\phi_{0}}\circ\varrho_{x}$ and hence $\xi_{c}^{\phi_{t}}(\pi_{\phi_{t}%
}\circ\vec{\gamma}^{t})=\xi_{c}^{\phi_{t}}(\underline{w}_{t}\circ\pi_{\phi
_{0}}\circ\varrho_{x})=\zeta_{c}^{t}(\pi_{\phi_{0}}\circ\varrho_{x})$. By
using Lemma \ref{LieCS} we conclude that%
\begin{align*}
\tfrac{d}{dt}\left.  \xi_{c}^{\phi_{t}}(\pi_{\phi_{t}}\circ\vec{\gamma}%
^{t})\right\vert _{t=0}  &  =\tfrac{d}{dt}\left.  \zeta_{c}^{t}(\pi_{\phi_{0}%
}\circ\varrho_{x})\right\vert _{t=0}=r\int_{\varrho_{x}}\int_{c}p(\dot{A}%
_{0},q^{\ast}\mathbb{F},\overset{(r-1)}{\ldots},q^{\ast}\mathbb{F})\\
&  =\left(  r\int_{c}p(v_{\mathbb{A}}(X),\mathbb{F},\overset{(r-1)}{\ldots
},\mathbb{F})\right)  _{x}\left(  \int_{\varrho_{x}}ds\right)  =-\mu
_{c}(X)_{x},
\end{align*}
and the result follows.
\end{proof}

As a consequence of the preceding Lemma we have the following

\begin{proposition}
\label{LiiftVector}For $X\in\mathrm{Lie}\mathcal{G}$ we have $X_{\mathcal{U}%
_{c}}=X_{N}+2\pi i(\rho_{c}(X_{N})+\mu_{c}(X))\partial_{\theta}$.
\end{proposition}

We have $\mathrm{curv}\Xi_{c}=d\rho_{c}=\sigma_{c}$ and $\iota_{X_{\mathcal{U}%
_{c}}}\Xi_{c}=2\pi i\mu_{c}(X)$. Hence $\mathrm{curv}_{\mathcal{G}}(\Xi
_{c})=\sigma_{c}+\mu_{c}=\varpi_{c}$. We conclude that $(\mathcal{U}_{c}%
,\Xi_{c})$ is a $\mathcal{G}$-equivariant prequantization bundle of
$(N,\varpi_{c})$. Hence we have proved the following

\begin{theorem}
\label{TeoremaA} Let $(p,\Upsilon)\in\mathcal{I}_{\mathbb{Z}}^{r}(G)$, $A_{0}$
be a background connection\ on $P$ and $c\colon C\rightarrow M$ a
$\mathcal{G}$-invariant map. These data determine an action of $\mathcal{G}$
on $\mathcal{U}_{c}=N\times U(1)\rightarrow N$ by $U(1)$-bundle automorphisms
$\Phi_{\phi}(x,u)\!=\!(\phi x,\exp(2\pi i\alpha_{\phi}(x))\!\cdot u)$ such
that the connection $\Xi_{c}=\theta-2\pi i\rho_{c}$ is $\mathcal{G}%
$-invariant, and the $\mathcal{G}$-equivariant curvature of $\Xi_{c}$ is
$\mathrm{curv}_{\mathcal{G}}(\Xi_{c})=\varpi_{c}$.

For every $X\in\mathrm{Lie}\mathcal{G}$ we have $X_{\mathcal{U}_{c}}%
=X_{N}+2\pi i(\rho_{c}(X_{N})+\mu_{c}(X))\partial_{\theta}$..

If $c=\partial u$ for a $\mathcal{G}$-equivariant map $u\colon U\rightarrow M$
and $s_{u}=-%
{\textstyle\int\nolimits_{u}}
Tp(\mathbb{A},\overline{A}_{0})$, then $\alpha_{\phi}(x)=s_{u}(\phi
x)-s_{u}(x)$. Hence $S_{u}=\exp(2\mathrm{\pi}i\cdot s_{u})$ determines a
$\mathcal{G}$-equivariant section of $\mathcal{U}_{c}\rightarrow N$. Moreover,
we have $\nabla^{\Xi_{c}}S_{u}=-2\mathrm{\pi}i\sigma_{u}\cdot S_{u}$, where
$\sigma_{u}=\int_{u}p(\mathbb{F})\in\Omega^{1}(N)$.
\end{theorem}

\begin{remark}
\emph{In place of the principal }$U(1)$\emph{-bundle }$\mathcal{U}_{c}$,\emph{
we can consider the }$\mathcal{G}$\emph{-equivariant Hermitian line bundle
}$\mathcal{L}_{c}=N\times\mathbb{C}\rightarrow N$\emph{ with the action}
$\Phi_{\phi}(x,z)=(\phi x,\exp(2\pi i\alpha_{\phi}(x))\cdot z)$ \emph{and}
$\Xi_{c}$\emph{ determines a hermitian connection }$\nabla^{\Xi_{c}}$\emph{ on
this bundle with }$\nabla^{\Xi_{c}}f=df-2\pi i\rho_{c}\cdot f$\emph{\ for
}$f\colon N\rightarrow\mathbb{C}$\emph{.}\bigskip
\end{remark}

Let $N_{1}$ be another connected and simply connected manifold in which
$\mathcal{G}$ acts. If $g\colon N_{1}\rightarrow N$ is a $\mathcal{G}%
$-equivariant map then $\mathbb{A}_{1}=(\mathrm{id}_{P}\times g)^{\ast
}\mathbb{A}$ is a $\mathcal{G}$-invariant connection on $P\times
N_{1}\rightarrow M\times N_{1}$ and conditions a) b) and c) are satisfied. If
$A_{0}$ is a background connection then we have a cocycle $\alpha_{1}$. The
following Proposition can be easily proved.

\begin{proposition}
\label{Subconjunto}We have $(\alpha_{1})_{\phi}(x)=\alpha_{\phi}(g(x))$,
$(\Phi_{1})_{\phi}(x,u)=\Phi_{\phi}(g(x),u)$ and $(\Xi_{1})_{c}=(g\times
\mathrm{id}_{U(1)})^{\ast}\Xi_{c}$ for $\phi\in\mathcal{G}$, $x\in N^{\prime}$
and $u\in U(1)$. In particular, the map $g\times\mathrm{id}_{U(1)}%
\colon(\mathcal{U}_{1})_{c}\rightarrow\mathcal{U}_{c}$ is $\mathcal{G}$-equivariant.
\end{proposition}

\subsection{Change of background connection}

The prequantization bundle $\mathcal{U}_{c}$ and the\ connection $\Xi_{c}$ are
defined using a background connection $A_{0}$. If $A_{0}^{\prime}$ is another
background connection then we have $Tp(\mathbb{A},\overline{A}_{0}^{\prime
})=Tp(\mathbb{A},\overline{A}_{0})+Tp(\overline{A}_{0},\overline{A}%
_{0}^{\prime})+dTp(\mathbb{A},\overline{A}_{0},\overline{A}_{0}^{\prime}%
),$\ with $Tp(\overline{A}_{0},\overline{A}_{0}^{\prime})=\mathrm{pr}%
_{1}^{\ast}Tp(A_{0},A_{0}^{^{\prime}})\in\Omega^{2r-1,0}(M\times N)$, and
hence
\begin{multline*}%
{\textstyle\int_{u}}
Tp(\mathbb{A},\overline{A}_{0}^{\prime})=%
{\textstyle\int_{u}}
Tp(\mathbb{A},\overline{A}_{0})+%
{\textstyle\int_{u}}
Tp(\overline{A}_{0},\overline{A}_{0}^{\prime})+%
{\textstyle\int_{u}}
dTp(\mathbb{A},\overline{A}_{0},\overline{A}_{0}^{\prime})\\
=%
{\textstyle\int_{u}}
Tp(\mathbb{A},\overline{A}_{0})+%
{\textstyle\int_{u}}
Tp(\overline{A}_{0},\overline{A}_{0}^{\prime})+d%
{\textstyle\int_{u}}
Tp(\mathbb{A},\overline{A}_{0},\overline{A}_{0}^{\prime})+(-1)^{d}\!%
{\textstyle\int_{\partial u}}
Tp(\mathbb{A},\overline{A}_{0},\overline{A}_{0}^{\prime}).
\end{multline*}
If $d=\!\dim U<2r-1$ we have $%
{\textstyle\int_{u}}
Tp(\overline{A}_{0},\overline{A}_{0}^{\prime})=0$ as $Tp(\overline{A}%
_{0},\overline{A}_{0}^{\prime})\!\in\!\Omega^{2r-1,0}(M\times N)$ and hence
\begin{equation}%
{\textstyle\int_{u}}
Tp(\mathbb{A},\overline{A}_{0}^{\prime})=%
{\textstyle\int_{u}}
Tp(\mathbb{A},\overline{A}_{0})+d%
{\textstyle\int_{u}}
Tp(\mathbb{A},\overline{A}_{0},\overline{A}_{0}^{\prime})+(-1)^{d}%
{\textstyle\int_{\partial u}}
Tp(\mathbb{A},\overline{A}_{0},\overline{A}_{0}^{\prime}). \label{backgr2}%
\end{equation}
Moreover, if $\dim U=2r-1$ then $%
{\textstyle\int_{u}}
Tp(\mathbb{A},\overline{A}_{0},\overline{A}_{0}^{\prime})=0$ and we have%
\begin{equation}%
{\textstyle\int_{u}}
Tp(\mathbb{A},\overline{A}_{0}^{\prime})=%
{\textstyle\int_{u}}
Tp(\mathbb{A},\overline{A}_{0})+%
{\textstyle\int_{u}}
Tp(A_{0},A_{0}^{\prime})-%
{\textstyle\int_{\partial u}}
Tp(\mathbb{A},\overline{A}_{0},\overline{A}_{0}^{\prime}). \label{backgr3}%
\end{equation}

The next Proposition shows that the action changes under a change of $A_{0}$,
but the corresponding prequantization bundles are isomorphic.

\begin{proposition}
\label{cambio back}Let $A_{0}^{\prime}$ be another background connection and
denote by $\mathcal{U}_{c}^{\prime},\Xi_{c}^{\prime}$ and $\alpha_{c}^{\prime
}$ the bundle, connection and action determined by $A_{0}^{\prime}$. If we
define $\beta_{c}=\int_{c}Tp(\mathbb{A},\overline{A}_{0},\overline{A}%
_{0}^{\prime})\in\Omega^{0}(N)$ then $\Xi_{c}^{\prime}=\Xi_{c}-2\pi
id\beta_{c}$ and $\alpha_{\phi}^{\prime}=\alpha_{\phi}+\phi^{\ast}\beta-\beta
$. The map $\Psi\colon\mathcal{U}_{c}\rightarrow\mathcal{U}_{c}^{\prime}$
$\Psi(x,u)=(x,\exp(2\pi i\beta_{c}(x)\cdot u)$ is a $\mathcal{G}$-equivariant
isomorphism of $U(1)$-bundles and $\Psi^{\ast}(\Xi_{c}^{\prime})=\Xi_{c}$.
\end{proposition}

\begin{proof}
It follows easily from the definitions and the equality $\rho_{c}^{\prime
}=\rho_{c}+d\beta_{c}$, which is a consequence of equation (\ref{backgr2}).
\end{proof}

\begin{remark}
\label{trivialization}\emph{We interpret this result in the following way.
}$(p,\Upsilon)\in I_{\mathbb{Z}}^{r}(G)$\emph{ and }$c\colon C\rightarrow
M$\emph{ determine a }$\mathcal{G}$\emph{-equivariant prequantization bundle
}$(\mathcal{U}_{c},\Xi_{c})$ \emph{for }$(N,\varpi_{c})$\emph{, and a
background connection }$A_{0}$\emph{ determines a global trivialization of
this bundle. In this sense, the prequantization bundle does not depend on
}$A_{0}$\emph{.}
\end{remark}

The situation is different for the section associated to the Chern-Simons
action, as using equation (\ref{backgr3}) we obtain the following

\begin{proposition}
If $c=\partial u$ for a $\mathcal{G}$-equivariant map $u\colon U\rightarrow M$
and $S_{u}$, $S_{u}^{\prime}$ are the sections associated to $A_{0}$ and
$A_{0}^{\prime}$\ then $\Psi\circ S_{u}=S_{u}^{\prime}\cdot\exp(2\pi i\int
_{u}Tp(A_{0},A_{0}^{\prime}))$.
\end{proposition}

Hence the section $S_{u}$ is not intrinsically\ determined by $(p,\Upsilon
)\in\mathcal{I}_{\mathbb{Z}}^{r}(G)$ and $u$. To explain this, note that if
$S$ is a section satisfying $\nabla^{\Xi_{c}}S=-2\mathrm{\pi}i\sigma_{u}\cdot
S$, any other section satisfying this condition is given by $\exp(ia)S$ for
$a\in\mathbb{R}$ constant. The background connection $A_{0}$ determines a
constant $a$ and another connection $A_{0}^{\prime}$ determines a different
constant $a^{\prime}$, and hence a different section.

\subsection{Change of polynomial and submanifold}

The action of $\mathcal{G}$ on $\mathcal{A}\times U(1)$ is defined by a map
$\Phi_{\alpha}(x,u)=(\phi x,\exp(2\pi i\alpha_{\phi}(x))\cdot u)$ where
$\alpha\colon\mathcal{G}\times N\rightarrow\mathbb{R}/\mathbb{Z}$ satisfies
the cocycle condition $\alpha_{\phi_{2}\phi_{1}}(x)=\alpha_{\phi_{1}%
}(x)+\alpha_{\phi_{2}}(\phi_{1}x)$. If $\alpha$ and $\alpha^{\prime}$ satisfy
the cocycle condition, then it is also satisfied by $-\alpha$ and
$\alpha+\alpha^{\prime}$, and $\Phi_{-\alpha}=\Phi_{\alpha}^{-1}$ and
$\Phi_{\alpha+\alpha^{\prime}}=\Phi_{\alpha}\cdot\Phi_{\alpha^{\prime}}$. In
terms of line bundles, if $\mathcal{L}^{\alpha}$ is the $\mathcal{G}%
$-equivariant\ line bundle associated to a cocycle $\alpha$, then
$\Phi_{-\alpha}$ corresponds to the dual bundle $\mathcal{L}^{-\alpha
}=(\mathcal{L}^{\alpha})^{\ast}$ and $\Phi_{\alpha+\alpha^{\prime}}$
corresponds to the tensor product $\mathcal{L}^{\alpha+\alpha^{\prime}%
}=\mathcal{L}^{\alpha}\otimes\mathcal{L}^{\alpha^{\prime}}$.

We denote by $\alpha_{c}^{\vec{p}}$ the action determined by $\vec
{p}=(p,\Upsilon)\in\mathcal{I}_{\mathbb{Z}}^{r}(G)$, $c\colon C\rightarrow N$
and by $(\mathcal{L}_{c}^{\vec{p}},\nabla_{c}^{\vec{p}})$ the $\mathcal{G}%
$-equivariant line bundle and connection determined by them. If $c\colon
C\rightarrow N$, $c^{\prime}\colon C^{\prime}\rightarrow N$ are two smooth
maps we define $-c\colon-C\rightarrow N$, where $-C$ is the manifold $C$ with
the opposite orientation and $c+c^{\prime}\colon C\sqcup C^{\prime}\rightarrow
N$. Then we have $\alpha_{-c}^{\vec{p}}=-\alpha_{c}^{\vec{p}}$ ,
$\alpha_{c+c^{\prime}}^{\vec{p}}=\alpha_{c}^{\vec{p}}+\alpha_{c^{\prime}%
}^{\vec{p}}$ and also $\rho_{-c}^{\vec{p}}=-\rho_{c}^{\vec{p}}$,
$\rho_{c+c^{\prime}}^{\vec{p}}=\rho_{c}^{\vec{p}}+\rho_{c^{\prime}}^{\vec{p}}%
$. We conclude that $(\mathcal{L}_{-c}^{\vec{p}},\nabla_{-c}^{\vec{p}%
})=((\mathcal{L}_{c}^{\vec{p}},\nabla_{c}^{\vec{p}}))^{\ast}$ and
$(\mathcal{L}_{c+c^{\prime}}^{\vec{p}},\nabla_{c+c^{\prime}}^{\vec{p}%
})=(\mathcal{L}_{c}^{\vec{p}},\nabla_{c}^{\vec{p}})\otimes(\mathcal{L}%
_{c^{\prime}}^{\vec{p}},\nabla_{c^{\prime}}^{\vec{p}})$.

In a similar way if $\vec{p}=(p,\Upsilon)$, $\vec{p}^{\prime}=(p^{\prime
},\Upsilon^{\prime})\in\mathcal{I}_{\mathbb{Z}}^{r}(G)$ then we have
$\alpha_{c}^{-\vec{p}}=-\alpha_{c}^{\vec{p}}$, $\alpha_{c}^{\vec{p}+\vec
{p}^{\prime}}=\alpha_{c}^{\vec{p}}+\alpha_{c}^{\vec{p}^{\prime}}$ and
$\rho_{c}^{-\vec{p}}=-\rho_{c}^{\vec{p}}$, $\rho_{c}^{\vec{p}+\vec{p}^{\prime
}}=\rho_{c}^{\vec{p}}+\rho_{c}^{\vec{p}^{\prime}}$. Hence $(\mathcal{L}%
_{c}^{-\vec{p}},\nabla_{c}^{-\vec{p}})=((\mathcal{L}_{c}^{\vec{p}},\nabla
_{c}^{\vec{p}}))^{\ast}\mathcal{\ }$and $(\mathcal{L}_{c}^{\vec{p}+\vec
{p}^{\prime}},\nabla_{c}^{\vec{p}+\vec{p}^{\prime}})=(\mathcal{L}_{c}^{\vec
{p}},\nabla_{c}^{\vec{p}})\otimes(\mathcal{L}_{c}^{\vec{p}^{\prime}}%
,\nabla_{c}^{\vec{p}^{\prime}})$.

If $\partial u=c-c^{\prime}$, by Theorem \ref{TeoremaA} $S_{u}=\exp
(-2\mathrm{\pi}i\cdot%
{\textstyle\int\nolimits_{u}}
Tp(\mathbb{A},\overline{A}_{0}))$ determines a $\mathcal{G}$-equivariant
section of unitary norm of $\mathcal{L}_{c-c^{\prime}}^{\vec{p}}%
\simeq\mathcal{L}_{c}^{\vec{p}}\otimes(\mathcal{L}_{c^{\prime}}^{\vec{p}%
})^{\ast}\simeq\mathrm{Hom}(\mathcal{L}_{c^{\prime}}^{\vec{p}},\mathcal{L}%
_{c}^{\vec{p}})$ and hence $\mathcal{L}_{c^{\prime}}^{\vec{p}}$ and
$\mathcal{L}_{c}^{\vec{p}}$ are isomorphic as $\mathcal{G}$-equivariant line bundles.

\begin{remark}
\emph{It is important to recall that in the preceding formulas we are using
the same background connection }$A_{0}$\emph{ (i.e. the same trivialization
(see Remark \ref{trivialization})) for all the bundles. If we use different
connections }$A_{0}$\emph{ and }$A_{0}^{\prime}$\emph{ for }$c$\emph{ and
}$-c$\emph{, we do not have} $\mathcal{L}_{c}^{A_{0}}=(\mathcal{L}_{-c}%
^{A_{0}^{\prime}})^{\ast}$\emph{, but we have a pairing }$\mathcal{L}%
_{c}^{A_{0}}\otimes(\mathcal{L}_{-c}^{A_{0}^{\prime}})^{\ast}\rightarrow
N\times\mathbb{C}$\emph{.}
\end{remark}

\section{Application to the space of connections}

In this section the constructions of Section \ref{SectionFirst} are applied to
the space of connections on a principal $G$-bundle $P\rightarrow M$ . First we
give the explicit expressions of the forms that appear in the prequantization
bundle. A background connection is simply an element $A_{0}\in\mathcal{A}$.
The $1$-form $\rho_{c}=\int_{c}Tp(\mathbb{A},\overline{A}_{0})\in\Omega
^{1}(\mathcal{A})$ that determines the connection $\Xi_{c}$ is given by
\[
(\rho_{c})_{A}(a)=r(r-1)\int_{c}\int_{0}^{1}p(a,A-A_{0},F_{t},\overset
{(r-2)}{\ldots},F_{t})tdt,
\]
with $A_{t}=tA+(1-t)A_{0}$ and $F_{t}$ the curvature of $A_{t}$.

The form $\beta_{c}=\int_{c}Tp(\mathbb{A},\overline{A}_{0},\overline{A}%
_{0}^{^{\prime}})$ that appears in the change of background connection is
simply given by $\beta_{c}(A)=\int_{c}Tp(A,A_{0},A_{0}^{^{\prime}})$ (this
follows from the tautological definition of $\mathbb{A)}$. Finally, if
$u\colon U\rightarrow M$ is a $\mathcal{G}$-invariant map such that
$c=\partial u$ then $s_{u}=-%
{\textstyle\int\nolimits_{u}}
Tp(\mathbb{A},\overline{A}_{0})$ is given by $s_{u}(A)=-%
{\textstyle\int\nolimits_{u}}
Tp(A,A_{0})$. Also we have $(\sigma_{u})_{A}(a)=r\int_{u}p(a,F,\overset
{(2r-1)}{\ldots},F)$ for $a\in T_{A}\mathcal{A}\simeq\Omega^{1}(M,\mathrm{ad}%
P)$.

\subsection{Action by gauge transformations}

Now we consider the action of the group $\mathcal{G}=\mathrm{Gau}P$ of gauge
transformations on $P\rightarrow M$. In this case, as $\mathcal{G}$ does not
act on $M$, we can consider any smooth map $c\colon C\rightarrow M$ with $\dim
C=2r-2$. We summarize the results in the following

\begin{theorem}
\label{TeoremaAGau}Let $P\rightarrow M$ principal $G$-bundle, $(p,\Upsilon
)\in\mathcal{I}_{\mathbb{Z}}^{r}(G)$, $A_{0}$ be a background connection\ on
$P$ and $c\colon C\rightarrow M$ be a smooth map with $C$ closed, oriented and
$\dim C=2r-2$. If $\mathcal{G}$ acts on $P$ by elements of $\mathrm{Gau}P$,
these data determine an action of $\mathcal{G}$ on $\mathcal{U}_{c}%
=\mathcal{A}\times U(1)\rightarrow\mathcal{A}$ by $U(1)$-bundle automorphisms
such that the connection $\Xi_{c}=\theta-2\pi i\rho_{c}$ is $\mathcal{G}%
$-invariant and the equivariant curvature of $\Xi_{c}$ is $\varpi_{c}$.

Furthermore, if $c=\partial u$ and $s_{u}(A)=-%
{\textstyle\int\nolimits_{u}}
Tp(A,A_{0})$, then $\alpha_{\phi}(x)=s_{u}(\phi x)-s_{u}(x)$. Hence
$S_{u}=\exp(2\mathrm{\pi}i\cdot s_{u})$ determines $\mathcal{G}$-equivariant
section of $\mathcal{U}_{c}\rightarrow\mathcal{A}$ and we have $\nabla
^{\Xi_{c}}S_{u}=-2\mathrm{\pi}i\sigma_{u}\cdot S_{u}$.
\end{theorem}

\begin{remark}
\emph{In the classical case of Chern-Simons theory considered in the
Introduction, any }$SU(2)$\emph{-bundle }$P$\emph{ over a 3-manifold is
trivial. Hence we can take }$A_{0}$ \emph{the connection corresponding to a
global section }$u\colon M\rightarrow P$. \emph{Then for }$p(X)=\frac{1}%
{8\pi^{2}}\mathrm{tr}(X^{2})$\emph{ we have}$\ s_{M}(A)=-%
{\textstyle\int\nolimits_{M}}
Tp(A,A_{0})=-\frac{1}{8\pi^{2}}\int_{M}\mathrm{tr}(A\wedge dA+\frac{2}%
{3}A\wedge A\wedge A)$\emph{, which coincides with the classical Chern-Simons
action.}
\end{remark}

\begin{remark}
\label{RSWdef}\emph{In \cite{RSW} the equation }$\alpha_{\phi}(A)=s_{u}(\phi
A)-s_{u}(A)=-%
{\textstyle\int\nolimits_{u}}
Tp(\phi A,A_{0})+%
{\textstyle\int\nolimits_{u}}
Tp(A,A_{0})$ \emph{is used to define the action }$\alpha_{\phi}$\emph{. To do
this it\ is necessary to express the manifold }$c\,\ $\emph{as the boundary of
another manifold }$u$\emph{ and to extend the connections on }$c$\emph{ to
}$u$\emph{. This can be done in dimension two, but this procedure cannot be
generalized to higher dimensions.}
\end{remark}

In \cite{flat} it is used\ a different construction of a bundle based in
Theorem \ref{Prop}. If $\mathcal{G}$ is the subgroup of gauge transformations
fixing a point of $P$, then $\mathcal{G}$ acts freely on $\mathcal{A}$ and
$\mathcal{A}\rightarrow\mathcal{A}/\mathcal{G}$ is a principal $\mathcal{G}%
$-bundle (see \cite{Donaldson}). If $\mathfrak{A}$ is a connection on
$\mathcal{A}\rightarrow\mathcal{A}/\mathcal{G}$ we define $\mathbb{A}%
(\mathfrak{A})\in\Omega^{1}(P\times\mathcal{A},\mathfrak{g})$ by
$\mathbb{A}(\mathfrak{A})(Y)=\mathbb{A}((\mathfrak{A}((\mathrm{pr}_{2})_{\ast
}\pi_{\ast}Y))_{P\times\mathcal{A}})$ for $Y\in T(P\times\mathcal{A})$. Then
the connection $\mathbb{A}-\mathbb{A}(\mathfrak{A})$ is projectable onto a
connection $\underline{\mathfrak{A}}$ on $(P\times\mathcal{A})/\mathcal{G}%
\rightarrow M\times\mathcal{A}/\mathcal{G}$. If we set $\lambda_{c}=\rho
_{c}+\mu_{c}(\mathfrak{A})$, then we have (see \cite{flat}) $d\lambda_{c}%
=\pi^{\ast}(\int_{c}p(F_{\underline{\mathfrak{A}}}))=\mathrm{curv}(\int
_{c}\chi_{\underline{\mathfrak{A}}})$. Hence we can apply Theorem \ref{Prop}
to the character $\int_{c}\chi_{\underline{\mathfrak{A}}}\in\hat{H}%
^{2}(\mathcal{A}/\mathcal{G})$ and we obtain a cocycle $a_{\phi}%
^{\mathfrak{A}}(A)=\int\nolimits_{\gamma}(\rho_{c}+\mu_{c}(\mathfrak{A}%
))-(\int_{c}\chi_{\underline{\mathfrak{A}}})(\pi\circ\gamma)$ that in theory
determines another bundle. But this bundle coincides with ours, as we have the following

\begin{proposition}
\label{Igualdad copy(1)}We have $a_{\phi}^{\mathfrak{A}}=\alpha_{\phi}$ for
any $\phi\in\mathcal{G}$.
\end{proposition}

\begin{proof}
We define the projections $\overline{q}\colon P\times\mathcal{A}%
\times\mathbb{R}\rightarrow P\times\mathcal{A}$. The homomorphism
$h\colon\mathbb{Z\rightarrow}\mathcal{G}$ $h(n)=\phi^{n}$ determines an action
of $\mathbb{Z}$ on $P$ and $\mathcal{A}$.

Let $\pi_{\mathcal{G}}\colon P\times\mathcal{A}\rightarrow(P\times
\mathcal{A})/\mathcal{G}$ and $\pi_{\phi}\colon P\times\mathcal{A}%
\times\mathbb{R}\rightarrow(P\times\mathcal{A}\times\mathbb{R})/\mathbb{Z}$
denote the projections. If $\phi\in\mathcal{G}$, $x\in N$, $\gamma$ is a curve
on $N$ joining $A$ and $\phi A$ and $\vec{\gamma}(s)=(\gamma(s),s)$ by
definition we have%
\begin{align*}
\alpha_{\phi}(x)  &  =\int\nolimits_{\gamma}\rho_{c}-\int_{c}\chi
_{\underline{\overline{q}^{\ast}\mathbb{A}}_{\phi}}(\pi_{\phi}\circ\vec
{\gamma}),\\
a_{\phi}^{\mathfrak{A},\mathcal{G}}(x)  &  =\int\nolimits_{\gamma}\rho
_{c}+\int\nolimits_{\gamma}\mu_{c}(\mathfrak{A}))-(%
{\textstyle\int_{c}}
\chi_{\underline{\mathfrak{A}}})(\pi_{\mathcal{G}}\circ\gamma).
\end{align*}
The connections $\overline{q}^{\ast}\mathbb{A}$ and $\overline{q}^{\ast}%
\pi_{\mathcal{G}}^{\ast}\underline{\mathfrak{A}}$ are $\mathbb{Z}$-invariant
and project onto connections $\underline{\overline{q}^{\ast}\mathbb{A}}$ and
$A_{2}=\underline{\overline{q}^{\ast}\pi_{\mathcal{G}}^{\ast}\underline
{\mathfrak{A}}}$ on $(P\times\mathcal{A}\times\mathbb{R})/\mathbb{Z}%
\rightarrow(M\times\mathcal{A}\times\mathbb{R})/\mathbb{Z}$ and by equation
(\ref{AAprima}) we have $\int_{c}\chi_{\underline{\overline{q}^{\ast
}\mathbb{A}}}(\pi_{\phi}\circ\vec{\gamma})=\int_{c}\chi_{A_{2}}(\pi_{\phi
}\circ\vec{\gamma})+\int_{\vec{\gamma}}\int_{c}Tp(\overline{q}^{\ast
}\mathbb{A},\overline{q}^{\ast}\pi_{\mathcal{G}}^{\ast}\underline
{\mathfrak{A}})$. But $\int_{\vec{\gamma}}\int_{c}Tp(\overline{q}^{\ast
}\mathbb{A},\overline{q}^{\ast}\pi_{\mathcal{G}}^{\ast}\underline
{\mathfrak{A}})=\int_{\gamma}\int_{c}Tp(\mathbb{A},\pi_{\mathcal{G}}^{\ast
}\underline{\mathfrak{A}})=\int_{\gamma}\int_{c}Tp(\mathbb{A},\pi
_{\mathcal{G}}^{\ast}\underline{\mathfrak{A}})^{2r-2,1}$, where the
bigaduation is the induced by the product structure on $M\times\mathcal{A}$.
We have $Tp(\mathbb{A},\pi_{\mathcal{G}}^{\ast}\underline{\mathfrak{A}%
})^{2r-2,1}\!=\!r\!\int_{0}^{1}p(\mathbb{A(\mathfrak{A})},F_{t},\ldots,F_{t})$
with $F_{t}\!=\!\mathbb{F}\!+\!td_{\mathbb{A}}\mathbb{A(\mathfrak{A}%
)\!+}\!\frac{t^{2}}{2}[\mathbb{A(\mathfrak{A})},\mathbb{A(\mathfrak{A})}]$. As
$\mathbb{\mathfrak{A}}$ comes from a connection on $\mathcal{A}\rightarrow
\mathcal{A}/\mathcal{G}$ we have $F_{t}^{2,0}=\mathbb{F}^{2,0}$ and hence
$\int_{\gamma}\int_{c}Tp(\mathbb{A},\pi_{\mathcal{G}}^{\ast}\underline
{\mathfrak{A}})=r\int_{\gamma}\int_{c}p(\mathbb{A(\mathfrak{A})}%
,\mathbb{F},\ldots,\mathbb{F})=-\int_{\gamma}\mu_{c}(\mathbb{\mathfrak{A}})$.
We conclude that we have $\alpha_{\phi}(x)=\int\nolimits_{\gamma}\rho_{c}%
+\int_{\gamma}\mu_{c}(\mathbb{\mathfrak{A}})-\int_{c}\chi_{A_{2}}(\pi_{\phi
}\circ\vec{\gamma})$.

Hence we should prove that $\int_{c}\chi_{A_{2}}(\pi_{\phi}\circ\vec{\gamma
})=\int_{c}\chi_{\underline{\mathfrak{A}}}(\pi_{\mathcal{G}}\circ\gamma)$.
This result follows in a similar way as in the proof of Proposition
\ref{subgroupDiscreto}. If $\overline{Z}\colon(P\times\mathcal{A}%
\times\mathbb{R}\mathbf{)}/\mathbb{Z}\rightarrow(P\times\mathcal{A}%
)/\mathcal{G}$ and $z\colon(\mathcal{A}\times\mathbb{R}\mathbf{)}%
/\mathbb{Z}\rightarrow\mathcal{A}/\mathcal{G}$ are the natural maps, then we
have $A_{2}=\overline{Z}^{\ast}\underline{\mathfrak{A}}$ and $\int_{c}%
\chi_{A_{2}}=z^{\ast}\int_{c}\chi_{\underline{\mathfrak{A}}}$. Hence
$(\int_{c}\chi_{A_{2}})(\pi_{\phi}\circ\vec{\gamma})=(\int_{c}\chi
_{\underline{\mathfrak{A}}})(z\circ\pi_{\phi}\circ\vec{\gamma})=(\int_{c}%
\chi_{\underline{\mathfrak{A}}})(\pi_{\mathcal{G}}\circ\gamma).$
\end{proof}

In particular, the action $a_{\phi}^{\mathfrak{A}}$ does not depend on the
connection $\mathfrak{A}$ chosen on $\mathcal{A}\rightarrow\mathcal{A}%
/\mathcal{G}$.

\subsection{Restriction to the moduli space of irreducible flat
connections\label{Irreducible}}

We denote by $\widetilde{\mathcal{A}}$ the space of irreducible connections.
Although $\mathrm{Gau}P$ does not act freely on $\widetilde{\mathcal{A}}$, the
isotropy group is the same $Z(G)$ (the center of $G$) for all $A\in
\widetilde{\mathcal{A}}$, and $\widetilde{\mathcal{A}}/\mathrm{Gau}P$ is a
differential manifold. If we define the group $\widetilde{\mathcal{G}%
}=\mathrm{Gau}P/Z(G)$ then $\widetilde{\mathcal{G}}$ acts freely on
$\widetilde{\mathcal{A}}$ and $\widetilde{\mathcal{A}}\rightarrow
\widetilde{\mathcal{A}}/\widetilde{\mathcal{G}}$ is a principal $\widetilde
{\mathcal{G}}$-bundle (see for example \cite{Donaldson} for details). In the
preceding section we have constructed a $\mathrm{Gau}P$-equivariant
prequantization bundle $\mathcal{U}_{c}\rightarrow\mathcal{A}$. If we restrict
it to $\widetilde{\mathcal{U}}_{c}=\widetilde{\mathcal{A}}\times
U(1)\rightarrow\widetilde{\mathcal{A}}$ we hope that it will define a
prequantization bundle over $\widetilde{\mathcal{A}}/\widetilde{\mathcal{G}}$,
but there is a problem: the action of $Z(G)$ on $\widetilde{\mathcal{U}}_{c}$
does not need to be trivial and $\widetilde{\mathcal{G}}$ does not act on
$\widetilde{\mathcal{U}}_{c}$. Or, in an equivalent way, $\widetilde
{\mathcal{U}}_{c}/\mathrm{Gau}P\rightarrow\widetilde{\mathcal{A}}%
/\mathrm{Gau}P$ is not a $U(1)$-bundle. \ If the action of $Z(G)$ on
$\mathcal{U}_{c}$\ is trivial then $\widetilde{\mathcal{G}}$ acts on
$\mathcal{U}_{c}$, and restricting $\widetilde{\mathcal{A}}$ to $\widetilde
{\mathcal{U}}_{c}/\widetilde{\mathcal{G}}\rightarrow\widetilde{\mathcal{A}%
}/\widetilde{\mathcal{G}}$ we obtain a bundle over the moduli space of
irreducible connections. This is the case for the trivial $SU(2)$-bundle over
a surface, as it is shown in \cite{RSW}. If the action of $Z(G)$ on
$\mathcal{U}_{c}$\ is not trivial, we can define $\widetilde{G}=G/Z(G)$ and
$\widetilde{P}=P/Z(G)\rightarrow M$, which is a principal $\widetilde{G}%
$-bundle. We also set $\widetilde{\mathbb{P}}=\mathbb{(}P/Z(G))\times
\widetilde{\mathcal{A}}$ which is also a principal $\widetilde{G}$-bundle. The
connection $\mathbb{A}\in\Omega^{1}(\mathbb{P},\mathfrak{g})$ induces a
connection $\widetilde{\mathbb{A}}$ on $\widetilde{\mathbb{P}}$ which is
invariant under the action of $\widetilde{\mathcal{G}}$ (see \cite{flat} for
details). The results of Section \ref{SectionFirst} can be applied to the
bundle $\widetilde{P}=P/Z(G)\rightarrow M,$ $N=\widetilde{\mathcal{A}}$ and
the $\widetilde{\mathcal{G}}$-invariant connection $\widetilde{\mathbb{A}}$ on
$\widetilde{\mathbb{P}}$, and we obtain a result analogous to Theorem
\ref{TeoremaAGau}, but we should take polynomials and characteristic classes
of $\widetilde{G}$ in place of $G$. If $(p,\Upsilon)\in\mathcal{I}%
_{\mathbb{Z}}^{r}(\tilde{G})$ and $c\colon C\rightarrow M$ with $\dim C=2r-2$,
we obtain $\varpi_{c}\in\Omega_{\widetilde{\mathcal{G}}}^{2}(\widetilde
{\mathcal{A}})$ and a $\widetilde{\mathcal{G}}$-equivariant prequantization
bundle $(\widetilde{\Xi}_{c}$ $\widetilde{\mathcal{U}}_{c})$ of $(\widetilde
{\mathcal{A}},\varpi_{c})$, and taking the quotient a $U(1)$-bundle
$\widetilde{\mathcal{U}}_{c}/\widetilde{\mathcal{G}}\rightarrow\widetilde
{\mathcal{A}}/\widetilde{\mathcal{G}}$.

We consider only one example. If $G=SU(2)$ then $\widetilde{G}=SO(3)$. Both
groups have the same Lie algebra $\mathfrak{su}(2)\simeq\mathfrak{so}(3)$. As
they are connected, they have the same Weil polynomials $I(SU(2))=I(SO(3))$,
but $I_{\mathbb{Z}}(SO(3))\varsubsetneq I_{\mathbb{Z}}(SU(2))$. For example
the second Chern polynomial $c_{2}\notin I_{\mathbb{Z}}(SO(3))$, but the first
Pontryagin polynomial $p_{1}=4c_{2}\in I_{\mathbb{Z}}(SO(3))$ (see
\cite[Formula 4.11]{DW}). If $c\colon C\rightarrow M$ is a map with $C$ a
closed surface, the pre-symplectic structure $\underline{\sigma}_{c}$ on the
moduli space of irreducible flat connections determined by $c$ and the second
Chern class may not be prequantizable. But $4\cdot\underline{\sigma}_{c}$ is
always prequantizable by the bundle associated to the first Pontryagin class.

Let $(p,\Upsilon)\in\mathcal{I}_{\mathbb{Z}}^{r}(\widetilde{G})$. If
$\widetilde{\mathcal{F}}\subset\widetilde{\mathcal{A}}\ $is the space of
irreducible flat connections, for $r\geq2$ we have $\widetilde{\mathcal{F}%
}\subset\mu_{c}^{-1}(0)$. In particular the restriction to $\widetilde
{\mathcal{F}}\times U(1)$ of the form $\Xi_{c}$ is $\mathcal{G}$-basic.
$\Xi_{c}$ projects onto a connection on $\widetilde{\mathcal{F}}%
/\widetilde{\mathcal{G}}\times U(1)\rightarrow\widetilde{\mathcal{F}%
}/\widetilde{\mathcal{G}}$ and we obtain a prequantization bundle of
$(\widetilde{\mathcal{F}}/\widetilde{\mathcal{G}},\underline{\sigma}_{c})$,
where $\underline{\sigma}_{c}$ is obtained from $\varpi_{c}$ by symplectic
reduction. For $r=2$ and $C=M$ a closed oriented surface, we obtain
$(\sigma_{M})_{A}(a,b)=2\int_{M}p(a,b)$, $(\mu_{M})_{A}(X)=-2\int_{M}p(X,F)$
and $(\rho_{M})_{A}(a)=\int_{M}p(A-A_{0},a)$, for $A\in\mathcal{A}$, $a,b\in
T_{A}\mathcal{A}\simeq\Omega^{1}(M,\mathrm{ad}P)$ and $X\in\mathrm{Lie}%
\mathcal{G}$. If $p\colon\mathfrak{g}\times\mathfrak{g}\rightarrow\mathbb{R}$
is a non-degenerate bilinear form, then $\sigma_{M}$ is a symplectic form and
the moment map can be identified with the curvature map $A\mapsto F$. Hence
they coincide with the symplectic structure and moment map defined in
\cite{AB1}. As commented in Remark \ref{RSWdef} in this case our bundle also
coincides with that of \cite{RSW}, and the connection $\Xi_{M}$ projects onto
a connection on the quotient bundle $(\widetilde{\mathcal{F}}\times
U(1))/\widetilde{\mathcal{G}}\rightarrow\widetilde{\mathcal{F}}/\widetilde
{\mathcal{G}}$. If $J$ is a complex structure on $M$, it induces a complex
structure on $\mathcal{A}$ and $\sigma_{M}$ is of type $(1,1)$. As
$\nabla^{\underline{\Xi}_{M}}$ is a unitary connection we conclude (see
\cite{Donaldson}) that it determines a holomorphic structure on $\underline
{\mathcal{L}}_{M}\rightarrow\widetilde{\mathcal{F}}/\widetilde{\mathcal{G}}$.
We have similar results when $\dim M>2$ and $c\colon C\rightarrow M$ is a map
with $\dim C=2$. If $c=\partial u$, the restriction of $S_{u}$ to
$\widetilde{\mathcal{F}}$ is a $\Xi_{c}$-parallel section as it satisfies
$\nabla^{\Xi_{c}}S_{u}=0$ because $(\sigma_{u})_{A}(a)=2\int_{u}p(a,F)=0$ if
$A\in\mathcal{F}$.

If $r\geq3$ we have $\sigma_{c}|_{\mathcal{F}}=0$, and in this case
$\underline{\Xi}_{c}$ is a flat connection, and hence its holonomy defines a
cohomology class in $H^{1}(\widetilde{\mathcal{F}}/\widetilde{\mathcal{G}%
},\mathbb{R}/\mathbb{Z})$ (see \cite{flat} for a generalization of this result
to arbitrary dimensions).

\subsection{The action of automorphisms}

Let $\mathrm{Aut}^{+}P$ be the group of automorphisms preserving the
orientation on $M$, and assume that $\mathcal{G}$ is a group acting on $P$ by
elements of $\mathrm{Aut}^{+}P$. In this case we cannot choose $c\colon
C\rightarrow M$ an arbitrary map because it should be $\mathcal{G}$ invariant.
We only consider the cases $C=$ $M$ (if $\partial M=0$) and $C=\partial M$.

\subsubsection{Base manifold closed\label{SectionAutClosed}}

When $M$ is a closed manifold of dimension $2r-2$, we can take $C=M$ and
$c=\mathrm{id}_{M}$, which clearly is $\mathcal{G}$-invariant. As a
consequence of Theorem \ref{TeoremaA} we obtain the following{}

\begin{theorem}
\label{TeoremaAaut} Let $P\rightarrow M$ principal $G$-bundle with $M$ closed,
oriented and $\dim M=2r-2$, $(p,\Upsilon)\in\mathcal{I}_{\mathbb{Z}}^{r}(G)$,
$A_{0}$ a background connection\ on $P$ and a group $\mathcal{G}$ acting on
$P$ by elements of $\mathrm{Aut}^{+}P$. These data determine an action of
$\mathcal{G}$ on $\mathcal{U}_{M}=\mathcal{A}\times U(1)\rightarrow
\mathcal{A}$ by $U(1)$-bundle automorphisms such that the connection $\Xi
_{M}=\theta-2\pi i\rho_{M}$ is $\mathcal{G}$-invariant and the equivariant
curvature of $\Xi_{M}$ is $\varpi_{M}$.
\end{theorem}

Theorem \ref{TeoremaAaut} extends to arbitrary bundles the results of
\cite{Andersen1,Andersen} for trivial bundles over a surface.

\subsubsection{Base manifold with boundary\label{SectionAutBound}}

Now we assume that $M$ is a compact oriented manifold of dimension $2r-1$ with
boundary $\partial M$. We chose $C=\partial M$ and $c=\mathrm{id}_{M}$. By
applying Theorem \ref{TeoremaA} we obtain the following

\begin{theorem}
\label{TeoremaAautBound}Let $P\rightarrow M$ principal $G$-bundle with $M$
compact and oriented with boundary $\partial M$ and $\dim M=2r-1$. If a group
$\mathcal{G}$ acts on $P$ by elements of $\mathrm{Aut}^{+}P$, $(p,\Upsilon
)\in\mathcal{I}_{\mathbb{Z}}^{r}(G)$ and $A_{0}$ is a background
connection\ on $P$, these data determine an action of $\mathcal{G}$ on
$\mathcal{U}_{\partial M}=\mathcal{A}\times U(1)\rightarrow\mathcal{A}$ by
$U(1)$-bundle automorphisms such that the connection $\Xi_{\partial M}%
=\theta-2\pi i\rho_{\partial M}$ is $\mathcal{G}$-invariant and the
equivariant curvature of $\Xi_{\partial M}$ is $\varpi_{\partial M}$.

Furthermore, $S_{M}=\exp(-2\mathrm{\pi}i\cdot%
{\textstyle\int\nolimits_{M}}
Tp(A,A_{0}))$ determines $\mathcal{G}$-equivariant section of $\mathcal{U}%
_{\partial M}\rightarrow\mathcal{A}$ and we have $\nabla^{\Xi_{\partial M}%
}S_{M}=-2\mathrm{\pi}i\sigma_{M}\cdot S_{M}$.
\end{theorem}

\section{Riemannian metrics and diffeomorphisms}

In this Section we apply our results to the space of Riemannian metrics and
the action of diffeomorphisms. One possible approach to do it is to apply the
results of Section \ref{SectionPrequantization} to the structures on the space
of metrics defined in \cite{natconn}, \cite{anomalies} and \cite{WP}. However,
we follow a different approach: we obtain the prequantization bundle by
pulling back the bundles on the space of connections using the Levi-Civita map.

If $M$ is a oriented manifold and we take $P=FM$, the group $\mathcal{G}%
=\mathrm{Diff}^{+}M$ of orientation preserving diffeomorphisms acts on $FM$ by
automorphisms. The Levi-Civita map $\mathrm{LC}\colon\mathfrak{Met}%
M\rightarrow\mathcal{A}$ which assigns to a Riemannian metric $g$ its
Levi-Civita connection $LC(g)=\omega^{g}$ is $\mathcal{G}$-equivariant. If we
denote by $p_{k}\in I_{\mathbb{Z}}^{2k}(GL(n,\mathbb{R}))$ the $k$-th
Pontryagin polynomial and by $\Upsilon_{k}\in H^{4k}(\mathbf{B}GL(n,\mathbb{R}%
))\simeq H^{4k}(\mathbf{B}O(n))$ the $k$-th Pontryagin class, then $p_{k}$ and
$\Upsilon_{k}$ are compatible. We fix a polynomial $p\in\mathbb{Z}%
[p_{1},\ldots,p_{n/2}]\subset I_{\mathbb{Z}}^{\bullet}(GL(n,\mathbb{R}))$ of
degree $2r$ and the corresponding characteristic class $\Upsilon\in
H^{4k}(\mathbf{B}GL(n,\mathbb{R}))$.

\subsection{Closed manifolds}

Let $M$ be a compact closed manifold of dimension $4r-2$. If we fix a
background connection $A_{0}$ on $FM$, we can apply the results of Section
\ref{SectionAutClosed} and we obtain a $\mathcal{G}$-equivariant
prequantization bundle $(\mathcal{U}_{M},\Xi_{M})$ of the equivariant form
$\varpi_{M}=\sigma_{M}+\mu_{M}\in\Omega_{\mathcal{G}}^{2}(\mathcal{A})$. Using
the Levi-Civita map we obtain $\mathcal{U}_{M}^{\prime}=\mathrm{LC}^{\ast
}\mathcal{U}_{M}$, $\Xi_{M}^{\prime}=\mathrm{LC}^{\ast}\Xi_{M}$ and
$(\mathcal{U}_{M}^{\prime},\Xi_{M}^{\prime})$ is a $\mathcal{G}$-equivariant
$\mathcal{G}$-equivariant prequantization bundle of $\varpi_{M}^{\prime
}=\mathrm{LC}^{\ast}\varpi_{M}\in\Omega_{\mathcal{G}}^{2}(\mathfrak{Met}M)$.
It can be seen that $\varpi_{M}^{\prime}=\sigma_{M}^{\prime}+\mu_{M}^{\prime}$
coincide with the presymplectic structure and moment map defined in
\cite{WP}.\ We study in detail the simplest case.

\subsubsection{Dimension 2}

Let $M$ be a closed surface and $p_{1}(X)=-\frac{1}{8\pi^{2}}\mathrm{tr}%
(X^{2})$ is the first Pontryagin polynomial. The symplectic reduction of
$(\mathfrak{Met}M,\varpi)$ is studied in \cite{WP}, and the result is that
$(\mu_{M}^{\prime})^{-1}(0)=\mathfrak{Met}^{\ast}M$ is the space of metrics of
constant curvature.

If $M$ has genus $\gamma>1$ and $\mathfrak{Met}_{-1}M$ is the space of metrics
of constant curvature $-1$ we have $\mathfrak{Met}_{-1}M\subset\mu^{-1}(0)$.
The connected component with the identity $\mathrm{Diff}_{0}M$ acts freely on
$\mathfrak{Met}_{-1}M$ and the Teichm\"{u}ller space of $M$ is defined by
$\mathcal{T}(M)=\mathfrak{Met}_{-1}M/\mathrm{Diff}_{0}M$, which as it is well
known (e.g. see \cite{Tromba}), is a manifold of real dimension $6\gamma-6$.
It is proved in \cite{WP} that the form obtained from $\sigma_{M}^{^{\prime}}$
by symplectic reduction is $\underline{\sigma}_{M}^{^{\prime}}=\frac{1}%
{2\pi^{2}}\sigma_{\mathrm{WP}}$, where $\sigma_{\mathrm{WP}}$ is the
symplectic form of the Weil-Petersson metric on $\mathcal{T}(M)$. We define
the the quotient bundle $\mathcal{W}_{M}=(\mathfrak{Met}_{-1}M\times
U(1))/\mathrm{Diff}_{0}M\rightarrow\mathcal{T}(M)$. As $\mathfrak{Met}%
_{-1}M\subset(\mu_{M}^{\prime})^{-1}(0)$ the connection $\Xi_{M}^{\prime}$ is
projectable onto a connection $\vartheta_{M}$ on $\mathcal{W}_{M}$. Moreover,
as $\underline{\sigma}_{M}^{\prime}$ is of type $(1,1)$ and $\vartheta_{M}$ is
a unitary connection, we conclude (e.g. see \cite{Donaldson}) that
$\nabla^{\vartheta_{M}}$ determines a holomorphic structure on the line bundle
$\mathcal{L}_{M}\rightarrow\mathcal{T}(M)$ associated to $\mathcal{W}_{M}$.

Furthermore, the first Pontryagin class determines the action on
$\mathcal{L}_{M}$ of the elements of $\mathrm{Diff}^{+}M$ not connected with
the identity, and hence an action of $\Gamma_{M}=\mathrm{Diff}^{+}%
M/\mathrm{Diff}_{0}M$ (the mapping class group of $M$) on $\mathcal{L}_{M}$
which preserves $\nabla^{\vartheta_{M}}$. We conclude that $(\mathcal{L}%
_{M},\nabla^{\vartheta_{M}})$ is a $\Gamma_{M}$-equivariant holomorphic
Hermitian prequantization bundle for $(\mathcal{T}(M),\frac{1}{2\pi^{2}}%
\sigma_{\mathrm{WP}})$.

Similar prequantization bundles are constructed for example
in\ \cite{Guillarmou}\ and in \cite{Wolpert} by different techniques. We note
that our construction is not specific of two dimensions and can be applied to
any manifold of dimension $4r-2$.

\subsection{Manifolds with boundary}

If $M$ is a compact manifold of dimension $4r-1$ with boundary, we can apply
the results of Section \ref{SectionAutBound} and we obtain a $\mathcal{G}%
$-equivariant prequantization bundle $(\mathcal{U}_{\partial M},\Xi_{\partial
M})$ of $\varpi_{\partial M}=\sigma_{\partial M}+\mu_{\partial M}\in
\Omega_{\mathcal{G}}^{2}(\mathcal{A})$. By using the Levi-Civita map we obtain
$\mathcal{U}_{\partial M}^{\prime}=\mathrm{LC}^{\ast}\mathcal{U}_{\partial
M},\Xi_{\partial M}^{\prime}=\mathrm{LC}^{\ast}\Xi_{\partial M}$ and
$(\mathcal{U}_{\partial M}^{\prime},\Xi_{\partial M}^{\prime})$\ is a
$\mathcal{G}$-equivariant prequantization bundle of $\varpi_{\partial
M}^{\prime}=\mathrm{LC}^{\ast}\varpi_{\partial M}\in\Omega_{\mathcal{G}}%
^{2}(\mathfrak{Met}M)$. Furthermore, we have the following

\begin{theorem}
If $M$ is a compact oriented manifold with boundary $\partial M$ and we define
$s(g)=-%
{\textstyle\int\nolimits_{M}}
Tp(\omega^{g},A_{0})$, then $S(g)=\exp(-2\mathrm{\pi}i\cdot%
{\textstyle\int\nolimits_{M}}
Tp(\omega^{g},A_{0}))$ determines $\mathcal{G}$-invariant section of
$\mathcal{U}_{\partial M}^{\prime}\rightarrow\mathfrak{Met}M$.
\end{theorem}

Hence we have found a Chern-Simons line for Riemannian metrics.

We note that the prequantization bundle on $\mathcal{T}(M)$ is defined in
\cite{Guillarmou} by using a similar Chern-Simons line in dimension 3. They
express the surface as the boundary of a $3$-manifold and they use a
definition of the bundle similar to that in \cite{RSW} for connections. As in
the case of connections, this procedure cannot be extended to higher dimensions.

\end{document}